\documentclass[11pt]{article}
\usepackage{amsmath,amssymb,a4wide,amsthm,color,graphicx,verbatim,caption,enumerate,multirow}

\usepackage[norelsize,ruled,vlined,linesnumbered]{algorithm2e}
\SetKwComment{Comment}{$\triangleright$\ }{}

\SetAlgoProcName{~}{anautorefname}

\usepackage{booktabs} 

\newtheorem{theorem}{Theorem}[section]

\newtheorem{rem}[theorem]{Remark}

\newenvironment{remark}{\begin{rem} \em }{\em \end{rem}}
\newtheorem{lemma}[theorem]{Lemma}
\newtheorem{proposition}[theorem]{Proposition}
\newtheorem{ex}[theorem]{Example}

\newtheorem{definition}[theorem]{Definition}

\newcommand{\st} {\ensuremath{|\;}}

\usepackage{color}
\definecolor{orange}{RGB}{250, 54, 0}
\definecolor{purple}{rgb}{0.75, 0.0, 1.0}

\author{Saliha Ferda Do\u{g}an\footnote{All authors contributed equally to this work. All are from Bilkent University, Department of Industrial Engineering, Ankara, 06800 Turkey} \thanks{ferda.dogan@bilkent.edu.tr}
\and \"Ozlem Karsu \thanks{ozlemkarsu@bilkent.edu.tr} 
\and Firdevs Ulus \thanks{firdevs@bilkent.com.tr}}

\title{An exact algorithm for biobjective integer programming problems}

\date{\today}

\begin{document}

	\maketitle

	\begin{abstract} 
		We propose an exact algorithm for solving biobjective integer programming problems, which arise in various applications of operations research. The algorithm is based on solving Pascoletti-Serafini scalarizations to search specified regions (boxes) in the objective space and returns the set of nondominated points. We implement the algorithm with different strategies, where the choices of the scalarization model parameters and splitting rule differ. We then derive bounds on the number of scalarization models solved; and demonstrate the performances of the variants through computational experiments both as exact algorithms and as solution approaches under time restriction. The experiments demonstrate that different strategies have advantages in different aspects: while some are quicker in finding the whole set of nondominated solutions, others return good-quality solutions in terms of representativeness when run under time restriction. We also compare the proposed approach with existing algorithms. The results of our experiments show the satisfactory behaviour of our algorithm, especially when run under time limit, as it achieves {better} coverage of the whole frontier with {a} smaller number of solutions compared to the existing algorithms.
		\medskip
		
		
		\medskip
	\end{abstract}

	\section{Introduction} \label{sect: intro}
	
	In many operations research applications such as scheduling, task assignment and transportation, the underlying problem can be modeled as an integer programming problem. Moreover, a vast amount of these problems require two (or more) criteria to be considered, leading to biobjective (multiobjective) integer programming problems. 
	
	In this study, we focus on biobjective integer programming problems (BOIP) and propose an algorithm that returns the whole set of nondominated points of these problems.  There are a number of solution approaches that have been designed for BOIP in the literature, most of which explore the objective (criterion) space by repetitively solving single objective optimization problems related to the BOIP, called scalarization problems (or simply, scalarizations). A scalarization is formulated by means of a real-valued scalarizing function of the objective functions of the BOIP, auxiliary scalar or vector variables and/or parameters (\cite{Ehrgott}). 
	
	There are several scalarizations proposed in the literature. The widely-used ones are the weighted sum scalarization (\cite{Zadeh,Koski1988}), the $\epsilon$-constraint scalarization (\cite{Haimes}) and the (weighted) Chebyshev scalarization (\cite{Bowman,Steuer}). Most of the recent algorithms in the literature solve these scalarizations or their modifications repetitively to find the set of nondominated solutions. Commonly used algorithms are the perpendicular search and the $\epsilon$-constraint algorithm, which are based on weighted sum scalarization and $\epsilon$-constraint scalarization, respectively (\cite{Haimes,Chankong2,Chalmet,Ladesma,hamacher2007finding,leitner2015computational}). Examples of algorithms using weighted Chebychev scalarizations are proposed by \cite{Ralphs} and \cite{Sayin}, where a modified version of the scalarization is used. There are also two-phase algorithms, which generate the nondominated points at the extreme points of the convex hull (called ``extreme supported" points) in the first phase and find the unsupported nondominated points by exploring the triangles defined by two consecutive supported nondominated points in the second phase (\cite{Ulungu,Przybylski}). Recently, the balanced box algorithm is proposed by \cite{Boland}, and a two-stage algorithm which combines the balanced box and $\epsilon$-constraint algorithms is discussed by \cite{Rui}. Similarly, \cite{leitner2016ilp} suggests a hybrid approach combining  $\epsilon$-constraint method and binary search in the objective space, which was previously discussed in \cite{Ladesma}.
	
	We propose an exact solution algorithm that finds the whole set of nondominated solutions to BOIPs. The algorithm is based on the Pascoletti-Serafini scalarization (\cite{Pascoletti1984}), which depends on two parameters, a reference point and a direction vector. It has been employed in many vector optimization algorithms, mainly designed for continuous (linear / nonlinear convex / non-convex) problems, see for example~\cite{benson,lohne,cvop,nonconvex}. It is known that many scalarization models, including the weighted sum and the $\epsilon$-constraint scalarizations, can be seen as a special case of the Pascoletti-Serafini scalarization (\cite{eichfelder}). Moreover, using this scalarization, it is possible to find the set of all nondominated points even if the corresponding multiobjective optimization problem is not convex (which is not possible, for instance, with the weighted sum scalarization) or even if it has an ordering cone different from the nonnegative cone (which is not the case, for instance, with the $\epsilon$-constraint scalarization). The main motivation of using the Pascoletti-Serafini scalarization in this work is to use its flexible structure in order to generate a `nice' representation of the nondominated set, quickly. Note that by choosing the reference point and the direction vector, one has control on the positions of newly found nondominated points, which has the potential to reach this goal. As we will elaborate later on, in {a} weighted sum scalarization, the weight vector can be used to control the positions of newly found nondominated points in a similar way. However, since such approaches can only generate unsupported points by iteratively making them locally supported (supported within a box); it may not be possible to reach such points in the early iterations, which may reduce representativeness under time limited implementations.
	
	We adapt the Pascoletti-Serafini scalarization model for biobjective integer programming settings and discuss different implementation strategies for the algorithm, which we call variants for short. We compare these variants with respect to the number of (mixed) integer programming problems (IPs) solved  and solution time. We also test the performances of the variants under time limit and report on the representativeness of the obtained solution sets using the (scaled) coverage error (\cite{Sayin2000,Ceyhan}). We also perform similar comparative tests with existing algorithms from the literature. 
	
	The main contributions of this work can be summarized as follows: 
	{(i) demonstrating the advantages and drawbacks of various implementations of  a generic box exploration framework based on different parameter choices and splitting rules,} (ii) providing an intuitive geometric interpretation on how to set the direction in Pascoletti-Serafini scalarization so as to obtain good results under time limit, and (iii) performing an extensive comparative analysis.

	The structure of the paper is as follows. In Section~\ref{sect: prelim} we give the preliminaries and the problem definition. In Section~\ref{sect: alg} we explain the algorithm and provide bounds on the number of (mixed) integer programming problems solved. We test the performance of the algorithm and report the results of our experiments in Section \ref{sec:CE}. We conclude our discussion in Section \ref{sec:Conc}.

	\section{Preliminaries and problem definition} \label{sect: prelim}
	
	A general biobjective integer programming problem is formulated as
	\begin{equation} \tag{$P$}
	\text{``min"}\{\,z(x)=(z_1(x),  z_2(x))^T \st x\in\mathcal{X}\subseteq\mathbb{Z}^n \,\},
	\label{P}
	\end{equation}
	where $z_i(\cdot)$, $i=1,2$ are integer-valued objective functions. The set $\mathcal{X}$ represents the feasible set in the decision space, and the set $\mathcal{Z}=\{z(x) \st x\in \mathcal{X}\}$ represents the feasible set in the objective space. Note that the quotation marks are used for min as there is no complete order in the two dimensional vector space.
	
	Throughout the paper we will use the following notation for vector inequalities:
	\begin{align*}
	z(x^{'})\leq z(x) & \iff z_i(x^{'})\leq z_i(x) \text{~for~} i \in \{1,2\}; \\
	z(x^{'})\lneq z(x) & \iff z(x^{'})\leq z(x) \text{~and~} z(x^{'})\neq z(x);\\
	z(x^{'})< z(x) & \iff z_i(x^{'})< z_i(x) \text{~for~} i \in \{1,2\}.
	\end{align*}\begin{definition}
		$z(x^{'}) \in\mathcal{Z}$ \textit{dominates (strictly dominates)} $z(x)\in\mathcal{Z}$ if $z(x^{'})\lneq z(x)$ ($z(x^{'})< z(x)$). If there exists no $x^{'}\in\mathcal{X}$ such that $z(x^{'})$ dominates (strictly dominates) $z(x)$, then $z(x)$ is \textit{nondominated (weakly nondominated)} and $x$ is \textit{efficient (weakly efficient)}. 
		
	\end{definition}
	
	The set of all nondominated vectors is denoted by $\mathcal{N}$. The ideal and nadir points of problem \eqref{P}, respectively, are as follows:
	\begin{align*} 
	s^0=\big(\min_{x\in\mathcal{X}} z_1(x), \min_{x\in\mathcal{X}} z_2(x) \big)^T, \:\:
	u^0=\big(\max_{z\in \mathcal{N}} z_1,  \max_{z\in \mathcal{N}} z_2 \big)^T.
	\end{align*}
	
	A lexicographic optimization problem with two objective functions is given by
	\begin{equation} \label{eq:lexmin}
	\text{lexmin}\{\,z_i(x), z_j(x) \st  x\in\mathcal{X}\,\},
	\end{equation}
	where $i,j \in \{1,2\}$ and $i\neq j$. Solving~\eqref{eq:lexmin} means solving the following two (single objective) optimization problems: First, $\text{min}\{\, z_i(x)\st x\in\mathcal{X} \,\}$, and given an optimal solution $x^{'}$ of the first model, $ \text{min}\{\, z_j(x)\st x\in\mathcal{X},\ z_i(x)=z_i(x^{'})\,\}$. Solving a lexicographic optimization yields an efficient solution.
	
	
	In general, scalarization models are solved in order to find (weakly) efficient solutions. Throughout, the Pascoletti-Serafini scalarization is employed with two parameters: a reference point $s \in \mathbb{R}^2$ and a direction $d\in \mathbb{R}_+^{2}\setminus\{0\}$. The model is as follows:
	\begin{equation} 
	\label{eq:pascoletti serafini}
	\min \big\{\, \alpha ~|~ x\in\mathcal{X}, \ z(x) \leq s+\alpha d, \:\alpha \in \mathbb{R} \big\}.
	\end{equation}
	
	\begin{lemma}[\cite{Pascoletti1984}]\label{lem:PS}
		If $(x^*, \alpha^*)$ is an optimal solution of \eqref{eq:pascoletti serafini} for some $ s\in \mathbb{R}^2$ and $d\in \mathbb{R}_+^{2}\setminus\{0\}$, then $x^*$ is weakly efficient.
	\end{lemma}
	
	\begin{lemma}\label{lem:Lem2}If $(x^*,\alpha^*)$ is an optimal solution of \eqref{eq:pascoletti serafini} for some $ s\in \mathbb{R}^2$ and $d\in \mathbb{R}_+^{2}\setminus\{0\}$ then $y^*=s+\alpha^*d$ and $ z(x^*) $ are equal in at least one component. 
	\end{lemma}
	
	\begin{proof}
		Assume to the contrary that $ z(x^*) $ and $y^*$ are different in both components. That is, $z(x^*)<y^*$. Hence there exists $ \bar{\alpha}<\alpha^* $ such that $z(x^*)\leq s+\bar{\alpha}d$, which contradicts the optimality of $(x^*,\alpha^*)$.	
	\end{proof}
	
	When a subset $\bar{\mathcal {N}}$ of $\mathcal{N}$ is found through an algorithm or a procedure, in order to measure how well $\bar{\mathcal{N}}$ represents the set of all nondominated points ($\mathcal{N}$), it is possible to use the `coverage error' that is introduced by \cite{Sayin2000}.  Similar measures are used in the literature to measure representativeness, for example the coverage gap measure used recently in \cite{Ceyhan}. Here, we provide the definition of coverage error for the special case where Chebyshev ($L^{\infty}$) metric  is used. Note that this metric provides the maximum possible error that can be observed in any of the objective function values. As mentioned by \cite{Sayin2000}, other $L^{p}$ metrics would tend to add up coordinate-wise distances, which would not be appropriate in most cases where the objective functions are of a different nature, such as cost and quality. In that sense, using the Chebyshev metric is the safest choice.
	
	We also introduce the scaled version as in~\cite{Ceyhan}. 
	
	\begin{definition}\label{defn_gap}
		Let $\bar{\mathcal {N}} \subseteq \mathcal{N}$ be a representative subset.	The \emph{coverage error of $\bar{\mathcal{N}}$ with respect to $n\in\mathcal{N}$} is $$CE({\bar{\mathcal{N}}},{n})=\min_{\bar{n}\in\bar{\mathcal{N}}}\left(\max\{|n_1-\bar{n}_1|,|n_2-\bar{n}_2|\}\right),$$
		the \emph{coverage error} of $\bar{\mathcal{N}}$ is	$$CE(\bar{\mathcal{N}})=\max_{n\in\mathcal{N}}CE({\bar{\mathcal{N}}},{n})$$
		and the \emph{scaled coverage error} of $\bar{\mathcal{N}}$ is $$SCE({\bar{\mathcal{N}}})=\frac{CE(\bar{\mathcal{N}})}{\max\{u^0_1-s^0_1,u^0_2-s^0_2\}},$$ where $u^0$ and $s^0$ are the nadir and the ideal points, respectively.
	\end{definition}
	
	In addition to the coverage error, we also use `hypervolume gap' as another measure of representativeness, see e.g. \cite{zitzler, Boland, leitner2016ilp}. We employ the metric as used in \cite{Boland}. In particular, for $\bar{\mathcal{N}}\subseteq \mathcal{N}$, the hypervolume $H(\bar{\mathcal{N}})$ of $\bar{\mathcal{N}}$ is computed as the area of the region $$\bigcup_{n\in\bar{\mathcal{N}}} \{z\in \mathbb{R}^2 \st n\leq z\leq u^0\},$$ where $u^0$ is the nadir point of the problem. The hypervolume of the true set $H(\mathcal{N})$ is higher than that of a representative subset. Hypervolume gap is the difference $H(\mathcal{N})-H(\bar{\mathcal{N}})$ and a subset with less hypervolume gap is considered as having better representativeness. Similar to the scaled coverage error, we consider a scaled version of this metric.
	
	\begin{definition} \label{defn:HVG}
		Let $\bar{\mathcal {N}} \subseteq \mathcal{N}$ be a representative subset.	The \emph{scaled hypervolume gap (SHG) of $\bar{\mathcal{N}}$} is $$SHG({\bar{\mathcal{N}}})=\frac{H({\mathcal{N}})-H(\bar{\mathcal{N}})}{H({\mathcal{N}})}.$$	
	\end{definition}
	
	\section{The algorithm}\label{sect: alg}
	
	Throughout the algorithm the search regions in the objective space are referred to as boxes. A box is defined by three points in the criterion space, namely the starting point $s$, the nondominated point $t$ which defines the first component of the starting point and the nondominated point $p$ which defines the second component of the starting point, and denoted as follows $$b(s, p, t) = \{\,y \in  \mathbb{R}^2 \mid s_1\leq y_1 \leq p_1, \ s_2\leq y_2 \leq t_2 \, \}.$$
	
	Note that  it is possible to define the box using only $ p $ and $ t $. However, we keep the starting point $ s $ in the definition as it is used in the scalarization models.
	
	The general idea of the algorithm can be described as follows. At the beginning, two sets namely $ \mathcal{N} $ and $\mathcal{B}$, are defined to denote the set of nondominated points and boxes to be investigated, respectively. For initialization, two corner points of the nondominated set are found by solving $ \text{lexmin}\{\,z_1(x),z_2(x) \st x\in\mathcal{X}\,\} $ and $ \text{lexmin}\{\,z_2(x),z_1(x) \st x\in\mathcal{X}\,\} $. Let the optimal objective function vectors of these models be $t^0$, $p^0$, respectively. We initialize $\mathcal{N}$ as $\{t^0, p^0\} $ and $\mathcal{B}$ as $\{b(s^0, p^0, t^0)\}  $, where  $s^0$ is the ideal point. Clearly, the initial box includes all nondominated points. See Figure \ref{fig:IB} for the illustration of the initial region. 
	
	\begin{figure}[h] 	
		\centering
		\includegraphics[scale=0.45]{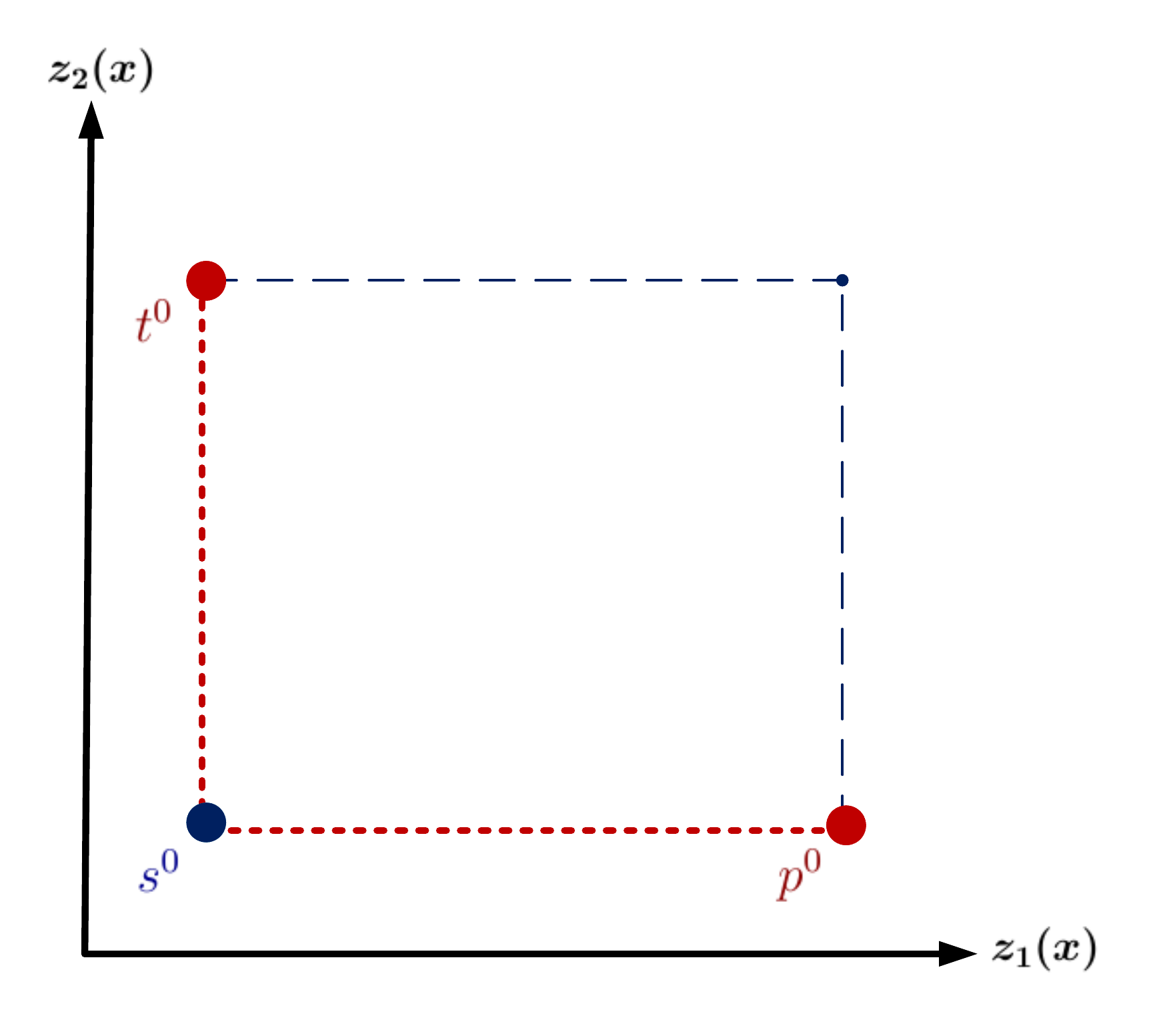}
		\caption{Initial box}
		\label{fig:IB}
	\end{figure}
	
	At each iteration, the algorithm searches one box from set $\mathcal{B}$ to find a (weakly) nondominated point by solving a Pascoletti-Serafini scalarization. In order to ensure finding a nondominated point, an (two) extra model(s) is (are) solved and the obtained nondominated point(s) is (are) added to $\mathcal{N}$. Then, the explored box is discarded and if at least one new nondominated point is found, two new boxes are added to $\mathcal{B}$ to be searched in the next iterations. The algorithm continues until there are no boxes to explore. The pseudocode of the algorithm is given by Algorithm~\ref{Alg}.
	
	At an arbitrary iteration, a box $b=b(s^b, p^b, t^b)$ from set $\mathcal{B}$ is selected and the following optimization problem is solved to search the box
	\begin{equation} \label{eq:Rbd}
	\tag{$R(b,d)$}
	\text{min}\{\, \alpha \st x\in\mathcal{X}, \;\alpha\in\mathbb{R}, \; z(x) \leq s^{b}+\alpha d, \; z_1(x) \leq \ p_1^{b} - \epsilon, \; z_2(x) \leq \ t_2^{b}- \epsilon\,\},
	\end{equation}
	where $d\geq 0$ is a direction vector in $ \mathbb{R}^2$ and $0<\epsilon<1$. This is a slightly modified Pascoletti-Serafini model. The last two constraints are added to prevent finding the (possibly weakly) nondominated points $p^b$ and $t^b$, which are already found in the previous iterations. If this problem is infeasible, then there is no nondominated point other than $p^b$ and $t^b$ in the box. Otherwise, let the optimal solution of~\eqref{eq:Rbd} be $(\alpha^b, x^b)$ and the corresponding (weakly) nondominated point be $n^{b}=z(x^b)$, see Lemma~\ref{lem:PS}. Note that $\alpha^b$ is the step size and defines the point $y^b =s^{b}+\alpha^b d$, which has at least one common component with $n^b$, see Lemma \ref{lem:Lem2}. Since the scalarization only guarantees that $n^b$ is weakly nondominated, the following problem(s) is (are) solved to ensure that a nondominated point is found. If the first components of $y^{b}$ and $n^b$ are equal ($n_1^{b}=y_1^{b}$) then,
	
	\begin{equation} \label{eq:P1}
	\tag{$P_1(x^b)$}
	\text{min}\{\, z_2(x)\st x\in\mathcal{X},\ z_1(x)=z_1(x^{b})\,\}
	\end{equation}
	is solved and, if the second components are equal ($n_2^{b}=y_2^{b}$) then,
	\begin{equation} \label{eq:P2}
	\tag{$P_2(x^b)$}
	\text{min}\{\, z_1(x)\st x\in\mathcal{X},\ z_2(x)=z_2(x^{b})\,\}
	\end{equation}
	is solved. Notice that it is possible to have $y^b=n^b$ and in this case, both problems are solved. Let the solutions of~\eqref{eq:P1} and~\eqref{eq:P2} be $x^{1}$ and $x^{2}$, respectively and $n^1=z(x^1)$ and $n^2=z(x^2)$ be the corresponding points in the criterion space. If only~\eqref{eq:P1} is solved, then $n^2$ is set to $n^b$ and symmetrically, if only~\eqref{eq:P2} is solved $n^1$ is set to $n^b$ (to be used in partitioning) (see lines 8-17 in Algorithm 1). Then, $\mathcal{N}$ is updated accordingly (lines 18-23).
	
	If both~\eqref{eq:P1} and~\eqref{eq:P2} are solved, it is possible to find two nondominated points $n^1$ and $n^2$ in the same iteration. In this case, both $n^1$ and $n^2$ are added to $\mathcal{N}$. See Figures \ref{fig:n1only2}-\ref{fig:n1n2} for illustrations of these cases. \par
	
	\begin{figure}[!] 
		\centering	
		\begin{minipage}[b]{0.49\textwidth}
			\includegraphics[scale=0.5]{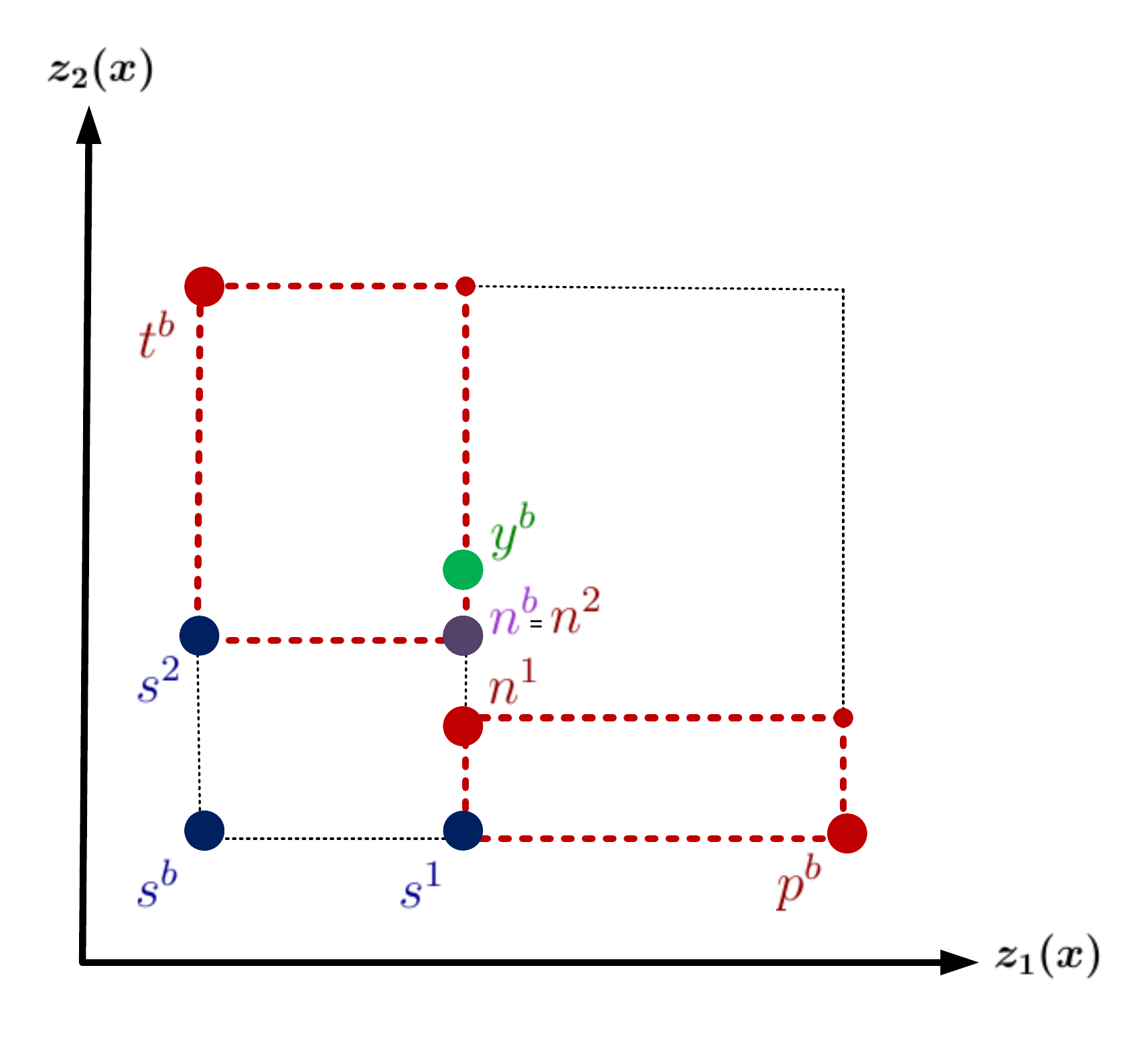}
			\caption{\eqref{eq:P1} is solved, $n^1$ is found as a nondominated point and $n^2$ is set to $n^b$.}
			\label{fig:n1only2}
		\end{minipage}
		\hfill
		\begin{minipage}[b]{0.49\textwidth}
			\includegraphics[scale=0.5]{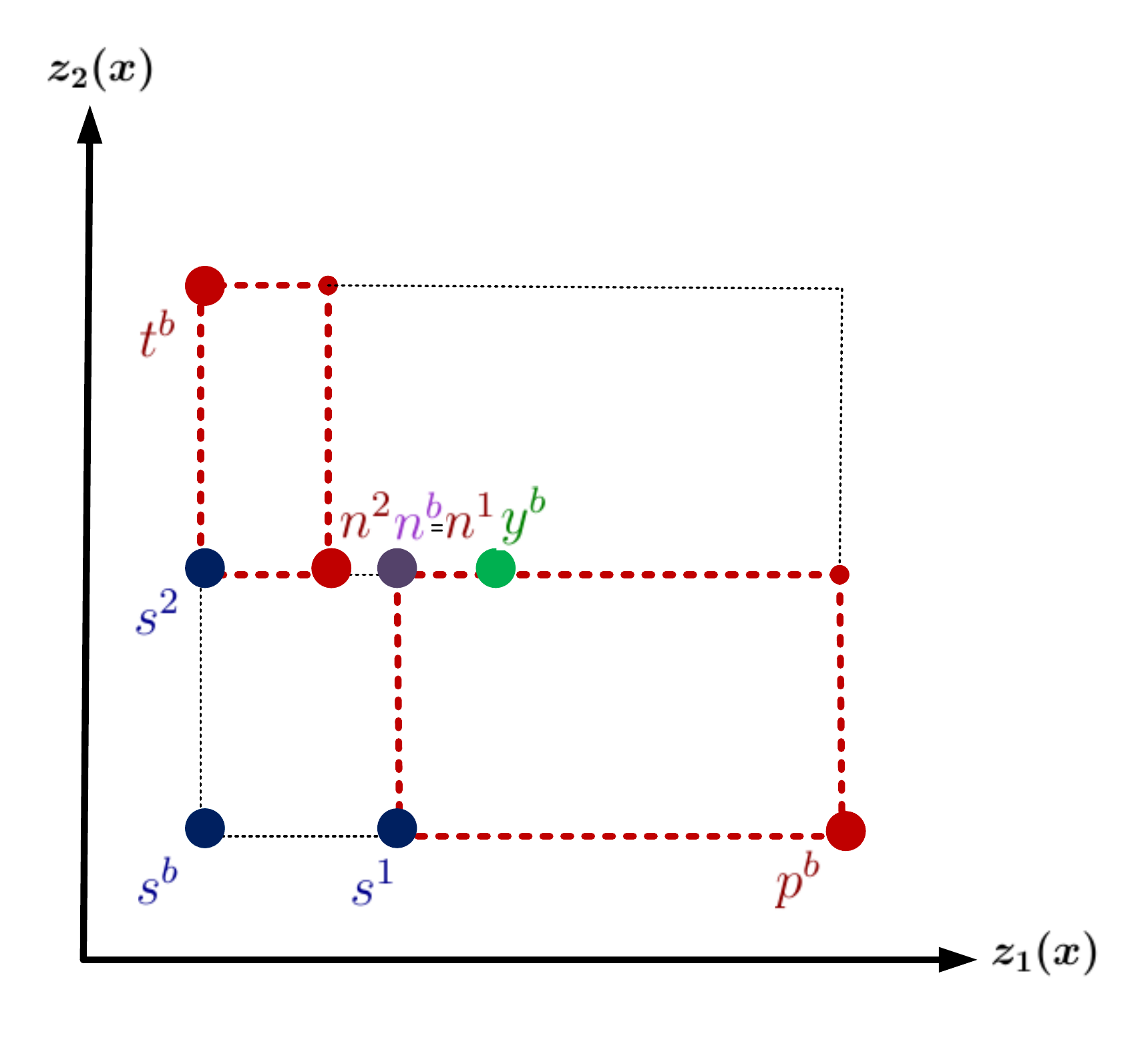}
			\caption{\eqref{eq:P2} is solved, $n^2$ is found as nondominated point and $n^1$ is set to $n^b$.}
			\label{fig:n2only2}
		\end{minipage}
		\vfill
		\begin{minipage}[b]{0.49\textwidth}
			\includegraphics[scale=0.5]{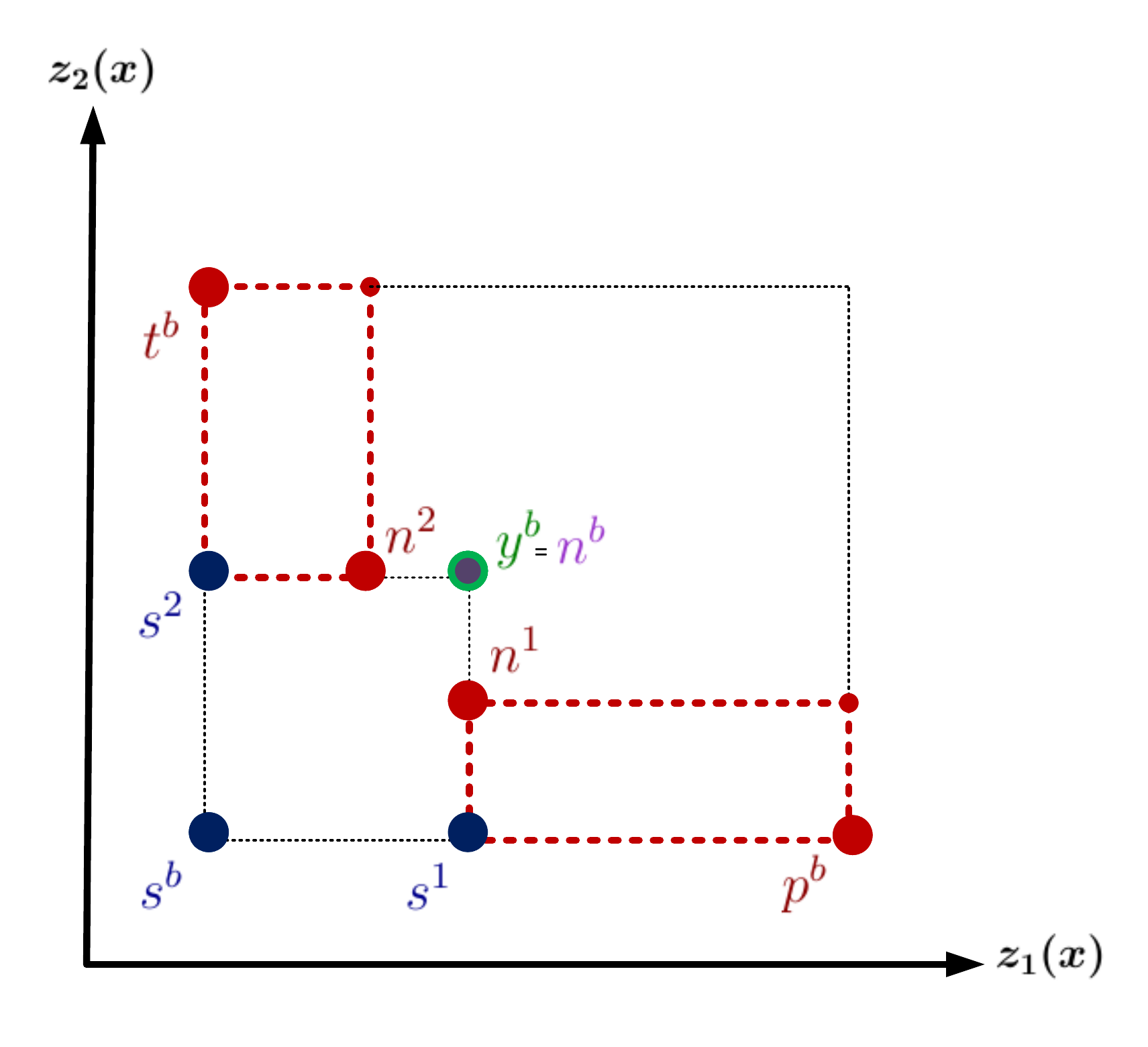}
			\caption{\eqref{eq:P1} and~\eqref{eq:P2} are solved, $n^1$ and $n^2$ are found as nondominated points.}
			\label{fig:n1n2}
		\end{minipage}
	\end{figure}
	
	
	
	For any (weakly) nondominated point $ n $, the dominated region $ \{y \in \mathbb{R}^2 ~|~ n \leq y\} $ and the dominating region $ \{y \in \mathbb{R}^2~|~ y \leq  n\} $  ($ \{y \in \mathbb{R}^2~|~ y < n\} $) can not contain any nondominated points; hence the current box $b(s^b, p^b, t^b)$ is split into two boxes using $n^1$ and $n^2$. More specifically, the first box is formed as $b(s^1, p^b, n^1)$, where $s^1=(n_1^1, p_2^b)^T$ and the second box is formed as $b(s^2, n^2, t^b)$, where  $s^2=(t_1^b, n_2^2)^T$. See Figures~\ref{fig:n1only2}-\ref{fig:n1n2} for the illustrations of newly formed boxes for different cases. \par
	
	Finally, the algorithm avoids searching regions which cannot have any new nondominated points, by taking the advantage of the integrality of the problem $(P)$ and the structure of a box. The boxes which do not satisfy $p_1^b-s_1^b>1$ and $t_2^b-s_2^b>1$ are eliminated since they can not include any nondominated points other than $p^b$ and $t^b$. After new boxes are defined and their sizes are checked to make sure that they can include nondominated points, they are added to set $\mathcal{B}$ to be searched in the next iterations. Then, the searched box $b(s^b, p^b, t^b)$ is removed from the set $\mathcal{B}$ (lines 24-30). The algorithm repeats the steps which are introduced above until there is no box in $\mathcal{B}$. 
	Note that the child boxes obtained by splitting a newly explored box are added to the end of the list of boxes, $\mathcal{B}$, (see lines 26-29) and the boxes are explored starting from the first box in  $\mathcal{B}$ (see line 4). Due to this structure, the algorithm always explores boxes obtained in previous iterations, before exploring the newly formed child boxes. This is to encourage exploration of the relatively larger boxes first. We, however, note that there might be exceptions: there might be cases, where a newly generated box is larger than a box generated in a previous iteration and hence explored later. We elaborate on this issue in Section \ref{subsec:Ext}.
	
	\begin{remark}\label{rem:weaknondom}
		Note that it is possible to have either $n^1$ (see Figure~\ref{fig:n2only2}) or $n^2$ (see Figure~\ref{fig:n1only2}) being not nondominated but only weakly nondominated. Hence for a box $b(s^b,p^b,t^b)$ considered through the algorithm, it is possible that $p^b,t^b$ are only weakly nondominated points. However, by the structure of forming the new boxes, even if $p^b$ ($t^b$) is not nondominated, there exists a nondominated point $p$ ($t$) such that $p^b_1=p_1$ ($t^b_2=t_2$). Clearly, the feasible region of~\eqref{eq:Rbd} model is the same as if the corresponding nondominated points $p, t$ are considered instead of the weakly nondominated ones. The only difference is in the selection of the reference point $s^b$. Because of this special structure, having boxes with weakly nondominated corners is not a problem in the sense that all the remaining nondominated points are included in the set of boxes to be searched.\end{remark}
	
	\begin{algorithm}[!] 
		
		\DontPrintSemicolon
		\fontsize{12pt}{12pt}\selectfont	
		\caption{The Proposed Algorithm for BOIP}
		\SetKwInOut{Input}{Input}
		\SetKwInOut{Output}{Output}	
		\Input{Problem \eqref{P}}
		\Output{The set of all nondominated solutions ($\mathcal{N}$)}
		\SetKwBlock{Initializations}{Initializations}{}
		
		\BlankLine
		\Initializations
		{		\begin{enumerate}[({I}1)]
				\item Set $d\geq 0$, $0<\epsilon<1$
				\item Solve $\text{lexmin}\{\,z_1(x),z_2(x)\st x\in\mathcal{X}\,\}$ 
				\Comment*[r]{to find the nondominated point $t^{0}$ } 			
				\item Solve $\text{lexmin}\{\,z_2(x),z_1(x)\st x\in\mathcal{X}\,\}$ \Comment*[r]{to find the nondominated point $p^{0}$ }			
				\item $\mathcal{N}=\{t^{0},p^{0}\}$, $s^{0}=(t^{0}_1, \ p^{0}_2)^T$, $\mathcal{B}=\{b(s^{0}, p^{0}, t^{0})\}$		
			\end{enumerate}
		} 	
		\fontsize{12pt}{12pt}\selectfont	
		\SetKwBlock{MainLoop}{MainLoop}{}
		\BlankLine
		\label{Alg}
		\MainLoop
		{
			\While{$\mathcal{B}$ is not empty}{		
				Let  $b(s^{b}, p^{b}, t^{b})\in \mathcal{B}$ and solve \eqref{eq:Rbd} \;
				\If {\eqref{eq:Rbd} is feasible}{  				
					$y^{b}=s^{b}+\alpha^{b}d$ \;
					$n^{b}=z(x^{b})$ \;	
					\eIf{$y^{b}_1=n^{b}_1$}{
						Solve  \eqref{eq:P1}. Let $ x^{1} $ be an optimal solution. \;
						$n^{1}=z(x^{1})$   \;				
					}{
						$n^{1}=n^{b}$ \;
					}
					\eIf{$y^{b}_2=n^{b}_2$}{
						Solve  \eqref{eq:P2} Let $ x^{2} $ be an optimal solution.\;
						$n^{2}=z(x^{2})$ \;				
					}{
						$n^{2}=n^{b}$ \;
					}	
					\If{$n^1_2<n^b_2$}{$\mathcal{N} \leftarrow \mathcal{N} \cup \{n^{1}\}$ \;}  	
					\If{$n^2_1<n^b_1$}{
						$\mathcal{N} \leftarrow \mathcal{N} \cup \{n^{2}\}$ \;
					}
					\If{$n^1_2 \geq n^b_2$ and $n^2_1 \geq n^b_1$}{
						$\mathcal{N} \leftarrow \mathcal{N} \cup \{n^{b}\}$ \;
					}	
					$s^{1}=(n^{1}_1, \ p^{b}_2)^T$  \Comment*[r]{first box $b(s^{1}, p^b, n^{1})$}   
					$s^{2}= (t^{b}_1, \ n^{2}_2)^T $ 
					\Comment*[r]{second box $b(s^{2}, n^{2}, t^b)$}	
					
					\If{$p^{b}_1-s^{1}_1 > 1$ and $n^{1}_2-s^{1}_2 > 1$}{
						
						$\mathcal{B} \leftarrow \mathcal{B} \cup \{b(s^{1}, p^b, n^{1})\}$}			
					\If{$n^{2}_1-s^{2}_1 > 1$ and $t^b_2-s^{2}_2 > 1 $ }{
						$\mathcal{B} \leftarrow \mathcal{B} \cup \{b(s^{2}, n^{2}, t^b)\}$  \;
					}		
				}
				$\mathcal{B}\gets \mathcal{B}\setminus \{b(s^{b}, p^{b}, t^{b})\} $ \;			
			}		
		}
	\end{algorithm}
	
	The algorithm works correctly and returns the set of all nondominated points after finitely many iterations. These are shown by the following two propositions.
	\begin{proposition}\label{prop:exact}
		Algorithm~\ref{Alg} works correctly: It returns the set of all nondominated points.
	\end{proposition}
	\begin{proof}
		The points that are added to set $\mathcal{N}$ are guaranteed to be nondominated. Indeed, \eqref{eq:Rbd} is a Pascoletti-Serafini scalarization with box contraints and by Lemma~\ref{lem:PS}, it returns a weakly efficient solution. By solving \eqref{eq:P1} and/or \eqref{eq:P2}, finding an efficient solution is guaranteed. Moreover, by the structure of defining the new boxes, see Remark~\ref{rem:weaknondom}, it is guaranteed that the set of all boxes to be searched ($\mathcal{B}$) includes all the remaining (if any) nondominated points at any time through the algorithm. 
	\end{proof}
	
	\begin{proposition} \label{Prop:UB}
		Algorithm 1 solves $(3N + C - 3C_2 -E -1 ) $ mixed integer programs, where $N=|\mathcal{N}|$ is the number of nondominated points, $C$ is the number of cases where $(y^b=n^b)$, $C_2$ is the number of sub-cases that two nondominated points are found and $ E $ is the number of eliminated boxes using the elimination rule. 
	\end{proposition}
	
	\begin{proof} The following expression, parts $(a)-(g)$ of which will be explained in detail, shows the number of models solved:
		\begin{equation*}\label{eq:nuofmodels}
		\underbrace{(4)}_{\textbf{(a)}}+\underbrace{(1)}_{\textbf{(b)}} + \underbrace{(2C_2)}_{\textbf{(c)}}  +\underbrace{2(N-2 -2C_2)}_{\textbf{(d)}}  +\underbrace{(N-2)}_{\textbf{(e)}} + \underbrace{(C-C_2)}_{\textbf{(f)}} - \underbrace{(E)}_{\textbf{(g)}}
		\end{equation*}
		At the beginning of Algorithm 1, two lexicographical   minimization problems are solved to find $t^0$ and $p^0$ \textbf{(a)} and one~\eqref{eq:Rbd} problem is solved to search the initial box at the first iteration of the while loop \textbf{(b)}. $2C_2$ points are found in $C_2$ number of cases ($y^b=n^b$ and two solutions are found), each of these points leads to a new box, hence a new~\eqref{eq:Rbd} model \textbf{(c)}. 
		For the rest of the nondominated points, ($ N-2C_2-2 $), each point results in two new boxes (and hence two~\eqref{eq:Rbd} models to be solved) \textbf{(d)}. As for the ($P_i(x^b)$) models: $ N $-2 points are found by solving a single second stage model (either~\eqref{eq:P1} or~\eqref{eq:P2}) \textbf{(e)}. Moreover, when $y^b=n^b$ and only a single nondominated point is found (in $ C-C_2 $ number of cases), we solve an extra~\eqref{eq:P1} or~\eqref{eq:P2}, which does not yield a new point \textbf{(f)}. Finally, $ E $ boxes are eliminated, avoiding the \eqref{eq:Rbd} models that would otherwise have been solved \textbf{(g)}.
	\end{proof}
	
	Note that the values for $C, C_2, E$ may not be deterministic even for a particular instance. Indeed, if the solver breaks ties arbitrarily whenever there are multiple optimal solutions, the order of the nondominated points found may change. This, of course, could affect the values of $C, C_2$ and $E$. Below, we provide the best- and  worst-case bounds depending only on $N$. 
	\begin{proposition}\label{prop:bounds}
		Without using the elimination rule, the upper and lower bounds on the number of mixed integer programs solved through the algorithm are $4N-3$ and $2N+1$, respectively.  
	\end{proposition}
	\begin{proof}
		By Proposition~\ref{Prop:UB}, Algorithm~\ref{Alg} solves $3N+C-3C_2-1$ mixed integer programs without the elimination rule. Moreover, by the definition of $C$ and $C_2$ we have 
		\begin{equation} \label{eq:ineq}
		C_2 \leq C \leq N-2-C_2.
		\end{equation}
		The last inequality follows by the fact that two nondominated points are found at the initialization step and $C_2$ of them are found additionally if two nondominated points are obtained after solving a Pascoletti-Serafini scalarization. The worst case occurs if $C_2=0$ and $C=N-2$, which yields the upper bound $4N-3$. 
		
		For the best case, it is required that $C$ takes its lowest possible value, which implies $C=C_2$. In this case the number of mixed integer programs can be written as $3N-2C_2-1$. Clearly, $C_2$ needs to take its highest possible value for the best case. By~\eqref{eq:ineq}, we have $C_2 \leq \frac{N-2}{2}$, hence the best case occurs if $C=C_2=\frac{N-2}{2}$, which yields the lower bound $2N+1$. 
	\end{proof} 
	Note that the number of eliminated boxes depends highly on the structure of the problem. In the worst case $E$ could be 0, while in the best case it could be as high as $N-1$. Consider an instance with $N=2^a+1$ for some integer $a$. If the nondominated points are located exactly on the integer diagonals of a $2^a \times 2^a$ plane, the number of eliminated boxes would be $N-1$. This can easily be seen by induction. 
	
	\subsection{An Alternative Splitting Strategy}\label{subec:Split}
	
	The new search regions added to $\mathcal{B}$ in each iteration can be chosen differently. In addition to the base version that is described above, we consider employing $y^b=s^b+\alpha^bd$ in defining the new regions. Accordingly, we use $y^b$ instead of $ n^b $ in order to define a new box if it yields a smaller region than the base version, see Figures \ref{fig:usingN} and \ref{fig:usingY}. This is done by replacing lines 12 and 17 of Algorithm 1 with $n^1 = y^b$ and $n^2 = y^b$, respectively. Notice that since the corners of the newly formed boxes are not necessarily integer valued, one also needs to change the elimination rule slightly. More specifically, for this variant the strict inequalities on lines 26 and 28 of Algorithm 1 are replaced by greater than or equal to signs. The following proposition shows this splitting strategy also works correctly. 
	\begin{proposition}\label{prop:yb}
		Algorithm 1 works correctly if $y^b= s^b+\alpha^bd$ is used in order to partition the boxes.
	\end{proposition}
	\begin{proof}
		Similar to the proof of Proposition~\ref{prop:exact}, $\mathcal{N}$ consists of nondominated points. To complete the proof, we show that it is not possible to eliminate a region which may contain a nondominated point. Note that there is no nondominated point in region $R:= \{y \st s^b \leq y < s^b+\alpha^bd\},$ where $(\alpha^b,x^b)$ is an optimal solution of \eqref{eq:Rbd} for some box $b$ and direction $d$. Indeed, for a feasible solution $\tilde{x}$ with $z(\tilde{x})\in R$, there exists $\tilde{\alpha}<\alpha^b$ such that $(\tilde{\alpha}, \tilde{x})$ is feasible for \eqref{eq:Rbd}, which contradicts the optimality of $(\alpha^b,x^b)$.	Moreover, if there is a nondominated point $n$ satisfying $n_i = y^b_i$, for $i\in \{1,2\}$, then it is found by solving \eqref{eq:P1} or \eqref{eq:P2}. Hence, no nondominated point is missed from the closure of $R$. 
		
		On the other hand, by Lemma \ref{lem:Lem2}, $y^b$ has at least one integer component satisfying $y^b_i=n^b_i$. Then, for any box $b$ through the algorithm, $t^b$ and $p^b$ has at most one non-integer component. More specifically, by the structure of defining the new boxes, it is ensured to have $p^b_1, t^b_2 \in \mathbb{Z}$ and there exist nondominated points, say $p, t$ such that $p^b_1 = p_1$ and $t^b_2 = t_2$. Therefore, for any $0<\epsilon<1$, no efficient solution from box $b$ is excluded from the feasible region of \eqref{eq:Rbd}.  
	\end{proof}

	\begin{figure}[h] 		
		\centering	
		\begin{minipage}[b]{0.49\textwidth}
			\includegraphics[scale=0.5]{box1FixedNp.pdf}
			\caption{Forming new boxes using $ n^b $}
			\label{fig:usingN}
		\end{minipage}
		\hfill
		\begin{minipage}[b]{0.49\textwidth}
			\includegraphics[scale=0.5]{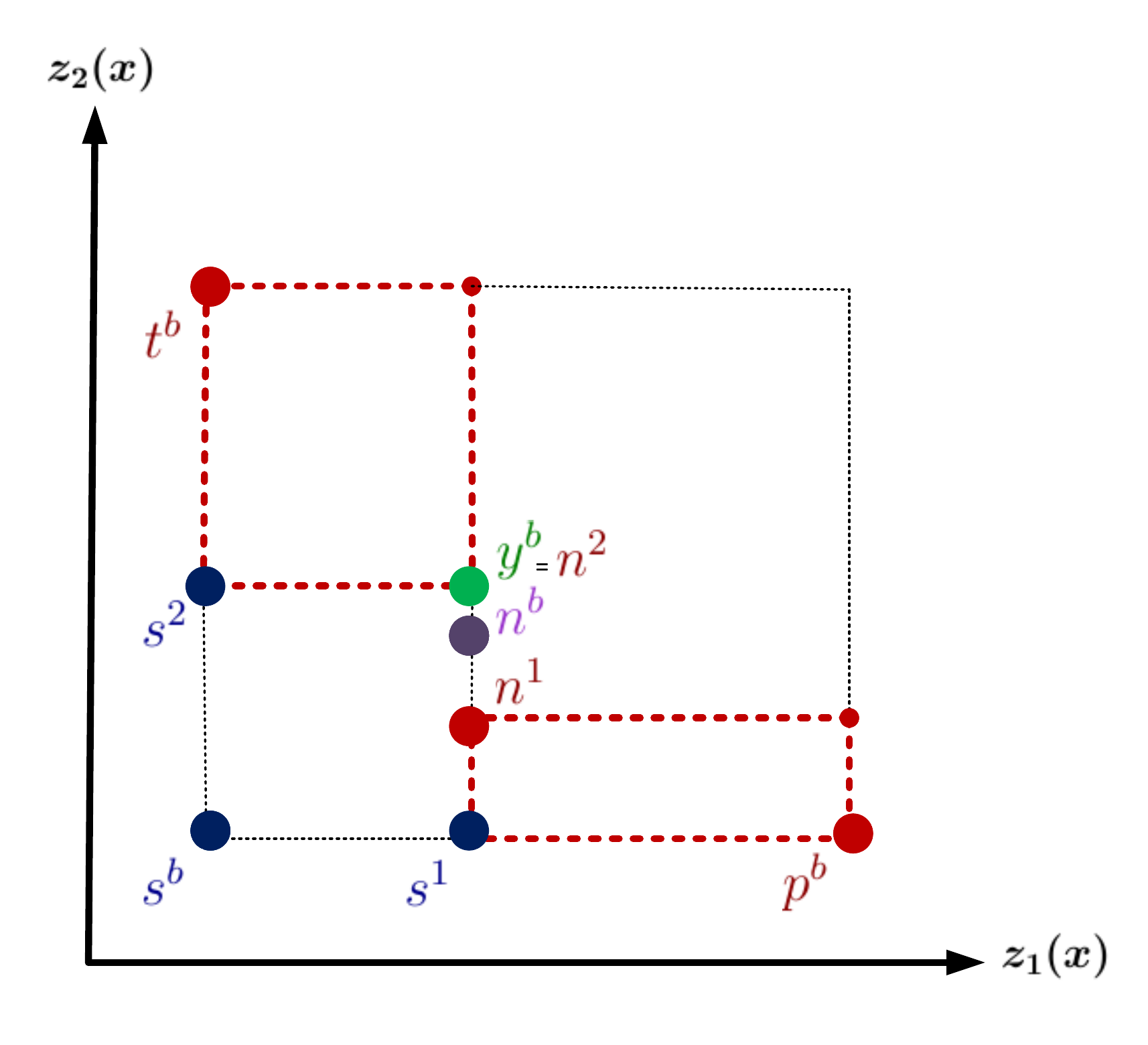}
			\caption{Forming new boxes using $ y^b $}
			\label{fig:usingY}
		\end{minipage} 		
	\end{figure}

	\section{Computational results} \label{sec:CE}
	
	The algorithm described in Section~\ref{sect: alg} can be implemented in different ways with respect to direction and splitting strategy choices. First of all, in each iteration the direction parameter $d$ can be fixed to $(1,1)^T$ \textbf{(Fixed)} or chosen according to the current box. We consider two alternatives for the latter: to set $d$ as the diagonal direction of the current box, $d = (p^b_1-s^b_1, t^b_2-s^b_2)^T$ \textbf{(Changing)} and to set $d$ as the direction starting from $s^b$ towards the nadir point, $d=(p^0_1-s^b_1, t^0_2-s^b_2)^T$  \textbf{(Nadir)}. Secondly, the splitting rule can be fixed as presented in Algorithm 1 or as explained in Section \ref{subec:Split}.
	
	The six combinations can be seen in Table~\ref{tab:var}. With a slight abuse of terminology, we refer to each such implementation as a variant of the algorithm.
	
	\begin{table}[h]
		\centering
		\caption{The variants of the algorithm}
		\begin{tabular}{|c|c|c|c|}
			\hline
			\multirow{2}{*}{Variants}& \textbf{Fixed} & \textbf{Changing} & \textbf{Nadir}\\ 
			& ~~~~~~~$d = (1,1)^T$~~~~~~~ & ~$d = (p^b_1-s^b_1, t^b_2-s^b_2)^T$ & ~$d=(p^0_1-s^b_1, t^0_2-s^b_2)^T$ \\ \hline
			\textbf{Using $n^b$} & \multirow{2}{*}{{FN}} & \multirow{2}{*}{{CN}} & \multirow{2}{*}{{NN}} \\ 
			(always) & & & \\ \hline 
			\textbf{Using $y^b$} &  \multirow{2}{*}{{FY}} & \multirow{2}{*}{{CY}} & \multirow{2}{*}{{NY}}\\
			(if smaller) & & & \\ \hline 
		\end{tabular}
		\label{tab:var}
	\end{table}

	We examine the efficiency of the algorithms by solving knapsack and assignment problem instances which are used in \cite{Boland}\footnote{The instances are available at http://
		hdl.handle.net/1959.13/1036183}. Both problem types contain four classes, A, B, C, D each with five instances. The first set consists of biobjective Knapsack Problem (KP) instances with 375, 500, 625 and 750 variables. The second set consists of biobjective Assignment Problem (AP) instances with 200 $\times$ 200 and $300 \times 300$ binary variables.\par
	
	The algorithms are coded in C++ and all mixed integer programming models are solved using CPLEX 12.6. 
	All of the instances are run on a computer with Intel Xeon CPU E5-1650 3.6 GHz processor and 64 GB RAM. Computation times are given in central processing unit (CPU) seconds. Unless otherwise stated, all performance measures are reported as average values over the five instances of each class.
	
	We first conduct preliminary experiments on type A knapsack and assignment instances in order to compare the performances of the algorithm variants.
	In Tables \ref{tab:PreKP} and \ref{tab:PreAP}, we report the average values for the number of nondominated points ($ N_{avg} $), the number of all (mixed) integer programming problems solved, the solution time (in CPU seconds), the number of \eqref{eq:Rbd} models solved, average time for solving one \eqref{eq:Rbd} model, average time for solving one ($P_i(x^b)$) model, $ C $,  $ C_2 $ and $E$, see Proposition \ref{Prop:UB}. 
	
	Overall, we see that partitioning a box using a nondominated point (e.g. $n^b$) is a better box defining strategy than using $y^b$. This leads to smaller number of problems solved, hence smaller solution times, except the KP case with changing direction according to nadir (see NN and NY in Table~\ref{tab:PreKP}). This good performance is mostly due to the increase in the number of  boxes that are eliminated (E) with our elimination rule (see lines 26 and 28 of Algorithm \ref{Alg}). 
	
	We observe that FN consistently performs good in terms of solution time over all test instances, being the fastest algorithm for KP and the second fastest for AP. 
	
	\begin{table}[htbp]
		\centering
		\caption{Comparison of alternative implementations for class A of the set KP}
		\resizebox{\textwidth}{!}{ 
			\begin{tabular}{|c|c|c|c|c|c|c|c|c|c|}
				\hline
				\multirow{2}[2]{*}{$ N_{avg} $} & \multirow{2}[2]{*}{Algorithm} & \multirow{2}[2]{*}{\# IP} & \multirow{2}[2]{*}{Run Time } & \multirow{2}[2]{*}{\# \eqref{eq:Rbd}} & \multirow{2}[2]{*}{Time per \eqref{eq:Rbd} } & \multirow{2}[2]{*}{Time per ($P_i(x^b)$)} & \multirow{2}[2]{*}{$ C $} & \multirow{2}[2]{*}{$C_2$} & \multirow{2}[2]{*}{$ E $} \\
				&       &       &       &       &       &       &       &       &  \\ \hline
				\multirow{6}[2]{*}{975.4} & FN   & 2541.20 & 838.06 & 1338.00 & 0.50  & 0.14  & 233.20 & 7.40  & 595.00 \\
				& FY   & 2762.80 & 947.70 & 1569.80 & 0.49  & 0.15  & 223.20 & 7.60  & 362.80 \\
				& CN   & 2398.60 & 894.72 & 1351.00 & 0.58  & 0.10  & 72.60 & 2.40  & 592.00 \\
				& CY   & 2512.60 & 1098.23 & 1520.40 & 0.62  & 0.14  & 15.20 & 0.40  & 426.60 \\
				& NN   & 2325.20 & 976.52 & 1347.60 & 0.65  & 0.10  & 0.20  & 0.00  & 600.20 \\
				& NY   & 2154.20 & 932.05 & 1176.80 & 0.67  & 0.14  & 0.00  & 0.00  & 771.00 \\
				\hline
		\end{tabular}}%
		\label{tab:PreKP}%
	\end{table}%
	
	\begin{table}[htbp]
		\centering
		\caption{Comparison of alternative implementations for class A of the set AP}
		\resizebox{\textwidth}{!}{ 
			\begin{tabular}{|c|c|c|c|c|c|c|c|c|c|}
				\hline
				\multirow{2}[2]{*}{$ N_{avg} $} & \multirow{2}[2]{*}{Algorithm} & \multirow{2}[2]{*}{\# IP} & \multirow{2}[2]{*}{Run Time } & \multirow{2}[2]{*}{\# \eqref{eq:Rbd} } & \multirow{2}[2]{*}{Time per \eqref{eq:Rbd} } & \multirow{2}[2]{*}{Time per ($P_i(x^b)$)} & \multirow{2}[2]{*}{$ C $} & \multirow{2}[2]{*}{$C_2$} & \multirow{2}[2]{*}{$ E $} \\
				&       &       &       &       &       &       &       &       &  \\ \hline
				\multirow{6}[0]{*}{708.4} & FN   & 1636.20 & 2150.36 & 699.20 & 2.29  & 0.58  & 246.60 & 20.00 & 674.60 \\
				& FY   & 1978.00 & 2764.41 & 1056.60 & 2.09  & 0.60  & 231.80 & 20.80 & 315.60 \\
				& CN   & 1553.40 & 2187.07 & 712.00 & 2.43  & 0.54  & 140.20 & 9.20  & 683.40 \\
				& CY   & 2206.40 & 3924.58 & 1372.00 & 2.52  & 0.58  & 132.40 & 8.40  & 25.00 \\
				& NN   & 1431.00 & 2043.09 & 720.60 & 2.33  & 0.50  & 0.00  & 0.00  & 693.20 \\
				& NY   & 1862.20 & 2931.79 & 1151.80 & 2.22  & 0.53  & 0.00  & 0.00  & 262.40 \\
				\hline
		\end{tabular}}%
		\label{tab:PreAP}%
	\end{table}%
	
	Based on these results, we conduct further preliminary experiments with FN, CN and NN variants. Since finding the whole set of nondominated points might be computationally demanding for most biobjective integer programming problems, early termination performances of the algorithms are also worth considering. Therefore, we run FN, CN and NN with predetermined time limits and report the quality of the set of nondominated points obtained. 
	
	Table \ref{tab:CG} shows the performance results for the three algorithm variants when they are run with time limits for class A of KP and AP. The time limit is set as 300 and 700 seconds for KP and AP, respectively. This corresponds to approximately 30\% of the time required to find the whole set of nondominated points. The table shows the average values of the number of nondominated points found ($\bar{N}$), the coverage error (CE), the scaled coverage error (SCE) and the scaled hypervolume gap (SGH) multiplied by $10^3$ for each variant.

	\begin{table}[htbp]
		\centering
		\caption{Representativeness results with time limits for class A instances}
		\resizebox{0.5\textwidth}{!}{ 	\begin{tabular}{|c|cccc|cccc|}
				\hline
				\multicolumn{1}{|c|}{} & \multicolumn{4}{c|}{KP} &  \multicolumn{4}{c|}{AP} \\\hline
				\multicolumn{1}{|c|}{\multirow{2}[1]{*}{Algorithm }} & \multirow{2}[1]{*}{$ \bar{{N}} $} & \multirow{2}[1]{*}{CE} & \multicolumn{1}{c}{\multirow{2}[1]{*}{SCE}} & \multicolumn{1}{c|}{\multirow{2}[1]{*}{SHG$\times 10^3$}} & \multirow{2}[1]{*}{$ \bar{{N}} $ } & \multirow{2}[1]{*}{CE} & \multicolumn{1}{c}{\multirow{2}[1]{*}{SCE}} & \multicolumn{1}{c|}{\multirow{2}[1]{*}{SHG$\times 10^3$}} \\					
				\multicolumn{1}{|c|}{} &  &     &       & \multicolumn{1}{c|}{} & \multicolumn{1}{c}{} &         &   & \multicolumn{1}{c|}{} \\\hline		
				FN    & 491   & 377.00 & 0.1086 &  8.1044	 &     227   & 719.20 & 0.2727 & 23.3229\\ 
				CN    & 530   & 15.00 & 0.0043 &    0.2715	  &  270   & 22.80 & 0.0086 &  0.3979 \\
				NN    & 365   & 561.00 & 0.1618 &  29.0612	  &    231   & 759.40 & 0.2879 & 29.0965 \\\hline
		\end{tabular}}%
		\label{tab:CG}
	\end{table}%

	It is seen that CN significantly outperforms the others with respect to representativeness. This result is expected as setting the direction as the diagonal vector of the box to be searched encourages the algorithm to find scattered solutions across the Pareto frontier and provides a highly representative set even at the early stages of the algorithm. 
	
	In Figures \ref{fig:KPFDN}-\ref{fig:KPNDN}, we provide the solution sets found when KP instances are solved with time limited versions of FN, CN and NN, respectively. Note that the approximation provided by CN outperforms the other two approximations, especially in terms of representing the tails of the frontier (top left and lower right). This result is expected and a clear benefit of choosing the direction vector based on the specifics of the box to-be-explored. To see why, consider an example case, where a box being explored is relatively wide in one dimension and narrow in the other. In that case, choosing a fixed direction vector may lead to finding a solution far away from the center of the box, as depicted in Figure \ref{fig:dispro}, where solving the scalarization with the fixed and changing direction policies will provide points 1 and 2, respectively. Clearly, 2 is a better representative subset of the set of solutions in the box, with respect to coverage error. It is seen in Figures \ref{fig:KPFDN}-\ref{fig:KPNDN} that such disproportionate boxes have to be explored to find solutions towards the tails of the Pareto frontier. When a time limit is applied and a subset of solutions are found with a direction vector different from the diagonal vector,  such solutions are not guaranteed to be close to the centers of the explored boxes, significantly reducing the representation quality of the subset.

	\begin{figure}[!]
		\centering	
		\begin{minipage}[b]{0.45\textwidth}
			\includegraphics[scale=0.45]{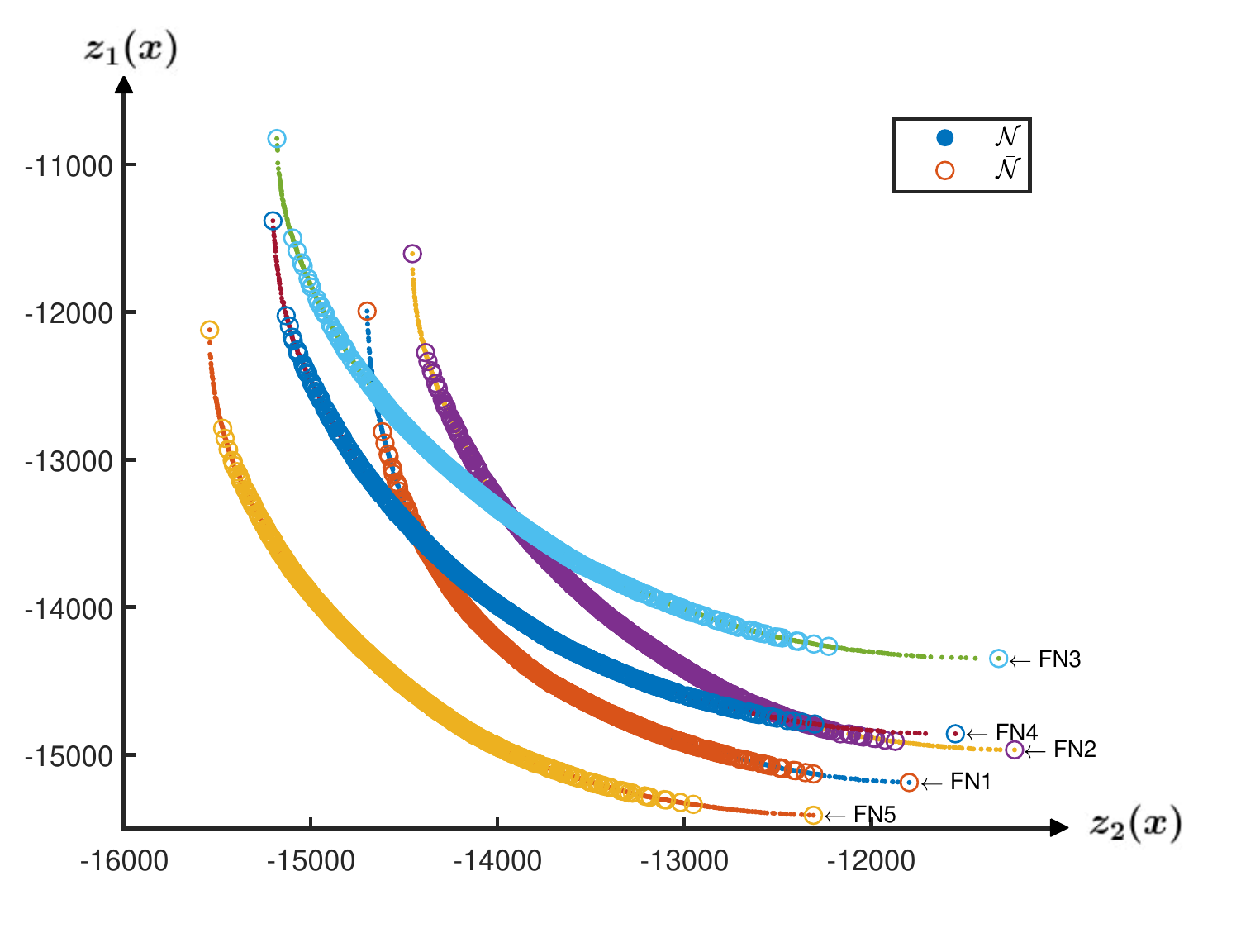}
			\caption{Solution sets found in time limited FN}
			\label{fig:KPFDN}
		\end{minipage}
		\hfill
		\begin{minipage}[b]{0.45\textwidth}
			\includegraphics[scale=0.45]{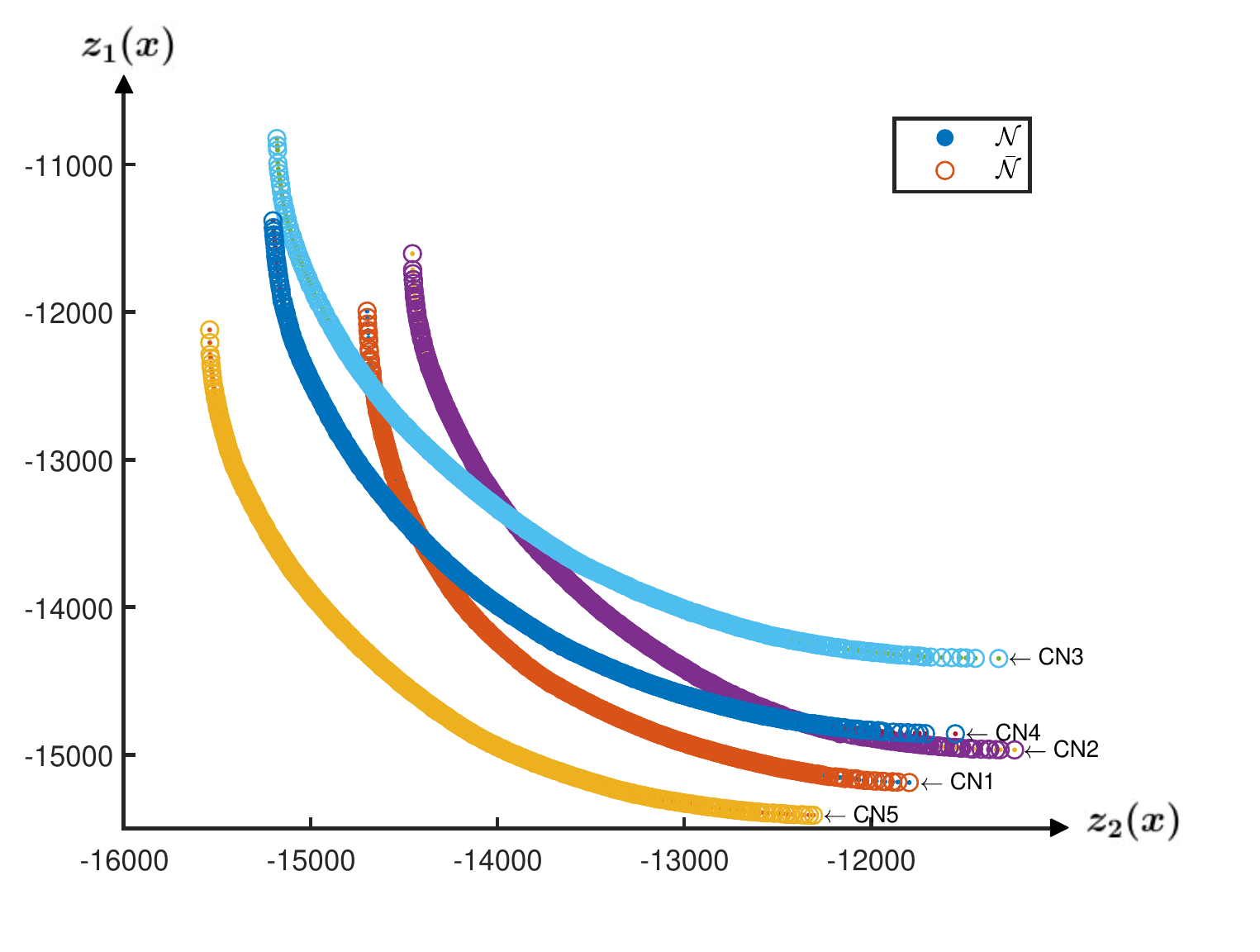}
			\caption{Solution sets found in time limited CN}
			\label{fig:KPCDN}
		\end{minipage}
		\vfill
		\begin{minipage}[b]{0.45\textwidth}
			\includegraphics[scale=0.45]{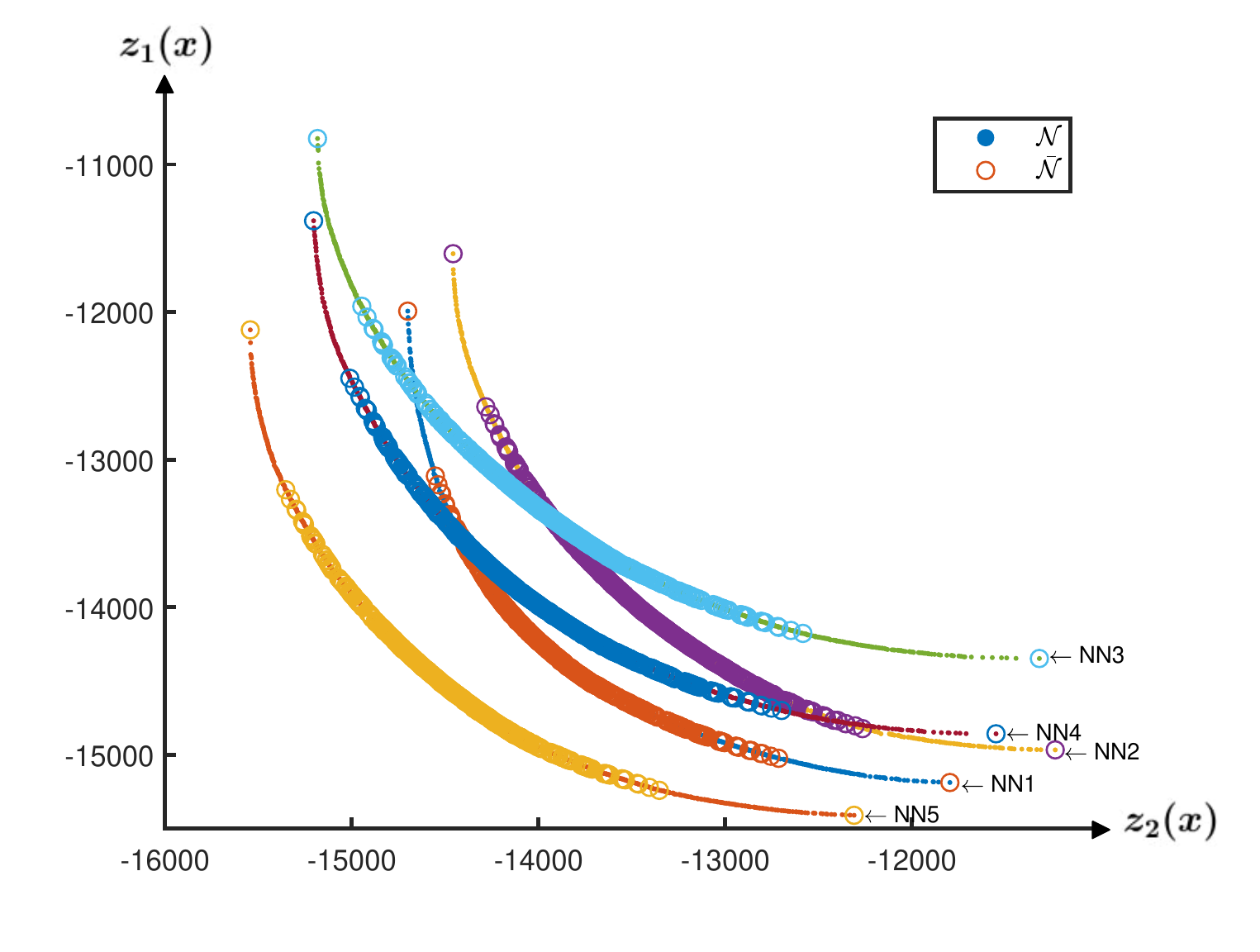}
			\caption{Solution sets found in time limited NN \vspace{1cm}}
			\label{fig:KPNDN}
		\end{minipage}
		\hfill
		\begin{minipage}[b]{0.45\textwidth}
			\includegraphics[scale=0.7]{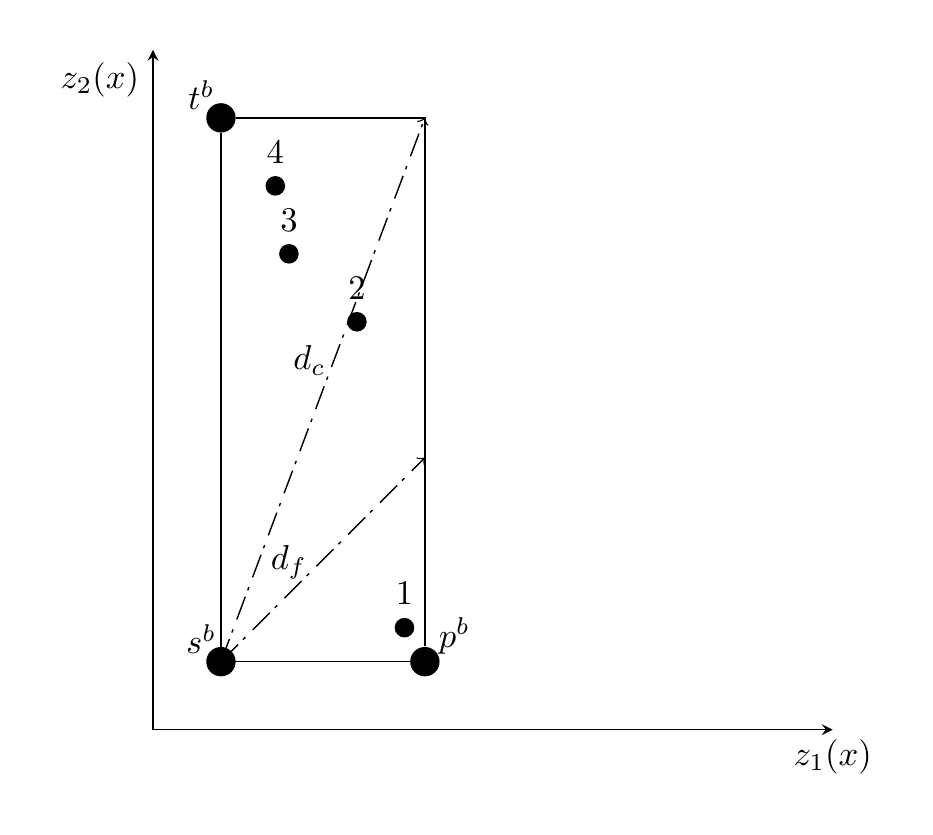}
			\caption{Effect of direction choice on a disproportional box $d_f=(1,1)$, $d_c$ is the diagonal vector}
			\label{fig:dispro}
		\end{minipage}
	\end{figure}

	In the first set of preliminary experiments, we eliminated the split strategy that defines the boxes using $y^b$, and concluded that FN, CN, NN are worth further consideration, FN being the most computationally efficient one. In the second set of experiments with time limits we have observed that CN is the top-performer. Based on these preliminary observations, we decided to perform the main experiments with FN as it is computationally more promising, and CN, as it outperforms the other variants under time restriction.
	
	Tables \ref{tab:MainKP} and  \ref{tab:MainAP}  show the results of our main experiments, in which we compare FN and CN over all instances of KP and AP. We report the average values for the number of nondominated solutions ($ N_{avg} $), the number of models solved, the total solution times, the number of \eqref{eq:Rbd} models solved and average time spent to solve \eqref{eq:Rbd} and  ($P_i(x^b)$) models as well as $C$, $C_2$ and $E$. The results verify the observations made at the preliminary experiments: although FN solves more (mixed) integer programming problems in total, it solves less of the more difficult \eqref{eq:Rbd} models, hence it works faster than CN. Moreover, when FN is used, the number of cases where $y^b=n^b$ ($C$) is significantly larger than that of CN. This is due to the nature of the direction vector used; moving along the search region in a fixed direction of $(1,1)^T$, the algorithm visits integer corners more often compared to a diagonal direction vector. This increases the cases where \eqref{eq:P1} and \eqref{eq:P2} are both solved within a box. Note that, in only a small portion of these cases two new nondominated points are found, implying that one ($P_i(x^b)$) is solved unnecessarily. However, since these models are much easier to solve compared to \eqref{eq:Rbd}, solving more of these does not significantly affect the computational performance of FN. 
	
	\begin{table}[htbp]
		\centering
		\caption{Results of the main experiments on KP}
		\resizebox{\textwidth}{!}{ 
			\begin{tabular}{cccccccccc}
				\multirow{2}[0]{*}{Class:$ N_{avg} $} & \multirow{2}[0]{*}{Algorithm} & \multirow{2}[0]{*}{\# IP} & \multirow{2}[0]{*}{Run Time} & \multirow{2}[0]{*}{\# \eqref{eq:Rbd}} & \multirow{2}[0]{*}{Time per \eqref{eq:Rbd} } & \multirow{2}[0]{*}{Time per ($P_i(x^b)$)} & \multirow{2}[0]{*}{$ C $} & \multirow{2}[0]{*}{$ C_2 $} & \multirow{2}[0]{*}{$ E $} \\
				&       &       &       &       &       &       &       &       &  \\\hline
				\multirow{2}[0]{*}{A:975.4} & FN   & 2541.20 & 838.06 & 1338.00 & 0.50  & 0.14  & 233.20 & 7.40  & 595.00 \\
				& CN   & 2398.60 & 894.72 & 1351.00 & 0.58  & 0.10  & 72.60 & 2.40  & 592.00 \\
				\multirow{2}[0]{*}{B:1539.4} & FN   & 3913.00 & 1546.16 & 1984.20 & 0.62  & 0.15  & 409.80 & 22.40 & 1046.80 \\
				& CN   & 3704.00 & 2711.98 & 2027.80 & 1.04  & 0.33  & 140.00 & 5.20  & 1037.60 \\
				\multirow{2}[0]{*}{C:2176.2} & FN   & 5453.60 & 2539.96 & 2665.00 & 0.76  & 0.18  & 657.20 & 46.80 & 1590.80 \\
				& CN   & 5152.20 & 3459.57 & 2744.00 & 1.12  & 0.16  & 239.40 & 9.40  & 1586.60 \\
				\multirow{2}[0]{*}{D:2791.8} & FN   & 6934.40 & 4605.52 & 3231.40 & 1.06  & 0.34  & 995.20 & 86.00 & 2177.20 \\
				& CN   & 6503.20 & 5404.77 & 3345.80 & 1.43  & 0.19  & 383.80 & 20.20 & 2194.40 \\
		\end{tabular}}%
		\label{tab:MainKP}%
	\end{table}%
	
	\begin{table}[htbp]
		\centering
		\caption{Results of the main experiments on AP}
		\resizebox{\textwidth}{!}{ 
			\begin{tabular}{cccccccccc}
				\multirow{2}[0]{*}{Class:$ N_{avg} $} & \multirow{2}[0]{*}{Algorithm} & \multirow{2}[0]{*}{\# IP} & \multirow{2}[0]{*}{Run Time} & \multirow{2}[0]{*}{\# \eqref{eq:Rbd}} & \multirow{2}[0]{*}{Time per \eqref{eq:Rbd} } & \multirow{2}[0]{*}{Time per ($P_i(x^b)$)} & \multirow{2}[0]{*}{$ C $} & \multirow{2}[0]{*}{$ C_2 $} & \multirow{2}[0]{*}{$ E $} \\
				&       &       &       &       &       &       &       &       &  \\\hline
				\multirow{2}[0]{*}{A:708.4} & FN   & 1636.20 & 2150.36 & 699.20 & 2.29  & 0.58  & 246.60 & 20.00 & 674.60 \\
				& CN   & 1553.40 & 2187.07 & 712.00 & 2.43  & 0.54  & 140.20 & 9.20  & 683.40 \\
				\multirow{2}[0]{*}{B:1416.2} & FN   & 3247.20 & 5354.08 & 1475.80 & 2.85  & 0.64  & 379.20 & 26.00 & 1301.60 \\
				& CN   & 3096.20 & 5519.86 & 1506.20 & 3.02  & 0.60  & 177.60 & 5.80  & 1311.60 \\
				\multirow{2}[0]{*}{C:823.6} & FN   & 1895.00 & 5644.20 & 803.60 & 5.22  & 1.32  & 288.80 & 23.00 & 794.60 \\
				& CN   & 1839.40 & 11212.29 & 815.80 & 12.04 & 1.33  & 210.60 & 12.60 & 803.20 \\
				\multirow{2}[0]{*}{D:1827} & FN   & 4140.20 & 16403.48 & 1808.20 & 6.95  & 1.64  & 561.40 & 58.40 & 1726.00 \\
				& CN   & 3980.40 & 17451.84 & 1860.40 & 7.61  & 1.54  & 304.00 & 13.00 & 1764.60 \\
		\end{tabular}}%
		\label{tab:MainAP}%
	\end{table}%

	CN works slower compared to FN but our preliminary experiments show that it is promising when used with time limits. We verified this observation by performing experiments for the whole KP and AP sets with time limit, the results of which are provided in Table \ref{tab:CGFull}.

	\begin{table}[htbp]
		\centering
		\caption{Representativeness results (average) with time limits for the full set of problem instances}
		\resizebox{\textwidth}{!}{ 
			\begin{tabular}{c|c|cccc|cccc|c|cccc|cccc}
				\multicolumn{10}{c}{~~~~~~~~~~~~~~~~~KP} & \multicolumn{9}{c}{AP} \\ \hline
				\multicolumn{6}{c}{~~~~~~~~~~~~~~~~~~~~~~FN} & \multicolumn{4}{c}{CN} & \multicolumn{5}{c}{~~~~~~FN} & \multicolumn{4}{c}{CN} \\ \hline
				\multirow{2}[0]{*}{Class} & \multirow{2}[0]{*}{Time} & \multirow{2}[0]{*}{$ \bar{{N}} $} & \multirow{2}[0]{*}{ CE} & \multirow{2}[0]{*}{ SCE} & \multirow{2}[0]{*}{ SHG$\times 10^3$} & \multirow{2}[0]{*}{$ \bar{
						{N}} $} & \multirow{2}[0]{*}{ CE} & \multirow{2}[0]{*}{ SCE} & \multirow{2}[0]{*}{ SHG$\times 10^3$} & \multirow{2}[0]{*}{Time} & \multirow{2}[0]{*}{$ \bar{{N}}$} & \multirow{2}[0]{*}{ CE} & \multirow{2}[0]{*}{ SCE} & \multirow{2}[0]{*}{ SHG$\times 10^3$} & \multirow{2}[0]{*}{$ \bar{{N}} $} & \multirow{2}[0]{*}{ CE} & \multirow{2}[0]{*}{ SCE} & \multirow{2}[0]{*}{ SHG$\times 10^3$} \\
				&       &       &       &       &       &       &       &       &       & & & & &       &       &       &       &  \\\hline
				
				A     & 300   & 491.00 & 377.00 & 0.1086 & 8.1044	& 530.00 & 15.00 & 0.0043 & 0.2715	 & 700   & 227.00 & 719.20 & 0.2727 & 23.3229	 & 269.80 & 22.80 & 0.0086 & 0.3979\\
				B     & 700   & 810.80 & 431.40 & 0.0914 & 5.9433	 & 789.60 & 16.80 & 0.0036 & 0.2579	 & 1820  & 479.60 & 2479.40 & 0.3131 & 17.7713	 & 621.40 & 36.60 & 0.0046 & 0.0990\\
				C     & 1000  & 998.00 & 594.40 & 0.0981 & 6.3360 & 814.80 & 21.40 & 0.0035 & 0.3443	 & 2810  & 380.20 & 571.00 & 0.2254 & 18.2876 & 246.20 & 40.00 & 0.0157 & 	0.8229\\
				D     & 1670  & 1345.80 & 705.20 & 0.0987 & 5.9066	      & 1074.40 & 18.60 & 0.0026 & 0.2374	 & 5650  & 617.20 &  3077.00 & 0.3161 & 17.5117	 & 764.00 & 54.60 & 0.0056 & 0.0841 \\ \hline
				
		\end{tabular}}%
		\label{tab:CGFull}%
	\end{table}%

	Overall, one can conclude that both variants are powerful in different aspects. When used to find the complete set of nondominated points, FN works better since it runs faster. However, CN is also very promising when run with a time limit since it quickly provides a highly representative subset of solutions.

	\subsection{Extensions} \label{subsec:Ext}
	In this section we present three extensions of Algorithm 1. The first of these extensions is specific to CN.  The other two can be implemented for all variants, however, we illustrate their effects also using CN, for demonstration purposes.
	
	When we examine the results of average time spent for an \eqref{eq:Rbd} model, we observe that there is significant difference between FN and CN for class C of AP, see Table \ref{tab:MainAP}. In these instances, average time spent per \eqref{eq:Rbd} model in CN is more than twice of the time spent in FN. To investigate this further, we check the solution times of each individual \eqref{eq:Rbd} model solved in CN for these instances. We summarize the results for class C instances in Figures 	\ref{fig:His} and \ref{fig:His_out}, where we report the distribution of runtimes with occurrence frequencies over all scalarization models solved in these instances (in total: 4079 models, 222 of which were infeasible). To improve visibility, the outliers (models with run time higher than 30 seconds) are excluded from Figure \ref{fig:His} and  are separately reported in Figure \ref{fig:His_out}. We see that the majority of the total time is occupied by only few models. 
	\begin{figure}[h] 		
		\centering	
		\begin{minipage}[b]{0.5\textwidth}
			\includegraphics[scale=0.33]{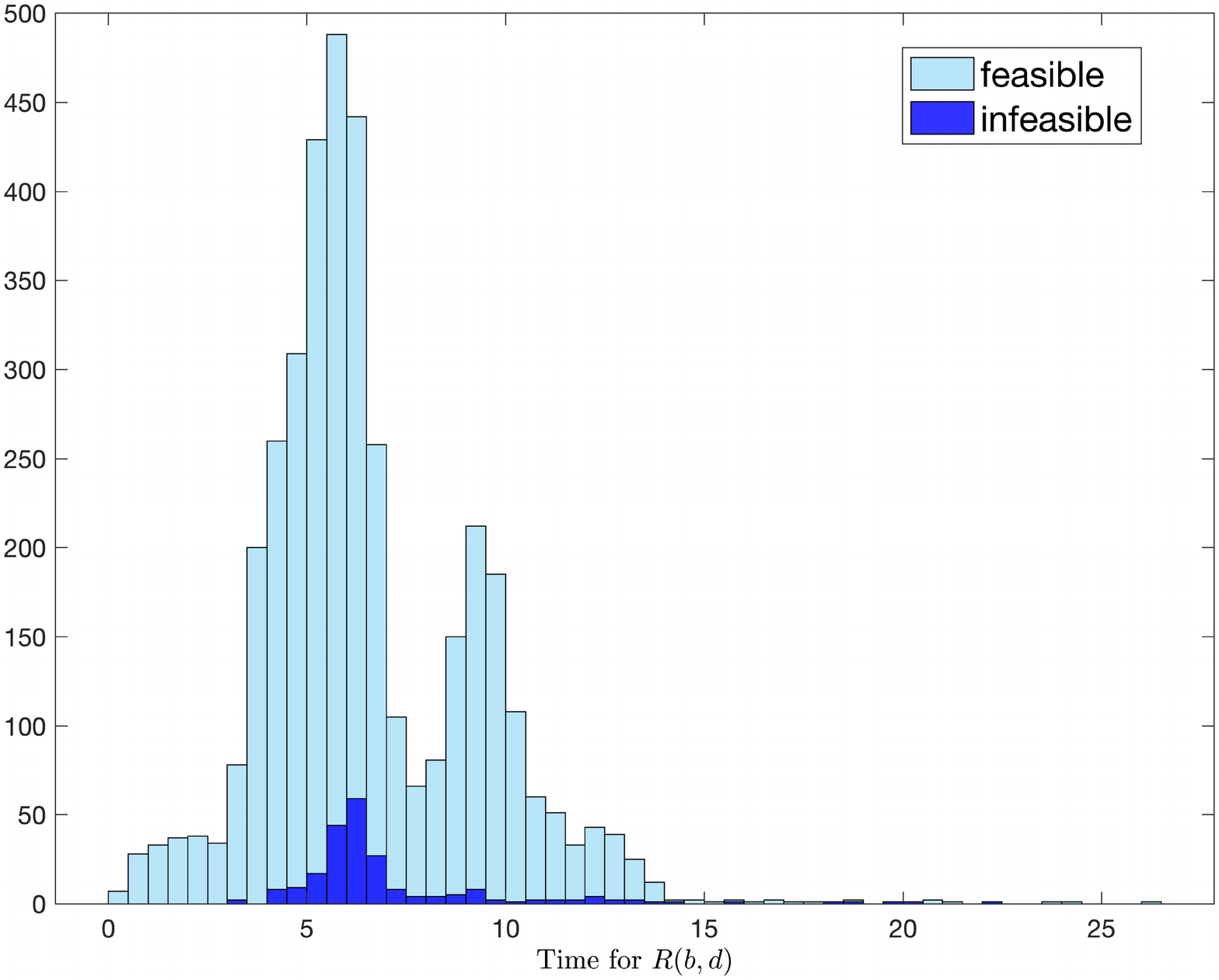}
			\caption{Occurrence frequencies of solution times that are under 30 seconds  }
			\label{fig:His}
		\end{minipage}
		\hfill
		\begin{minipage}[b]{0.5\textwidth}
			\includegraphics[scale=0.33]{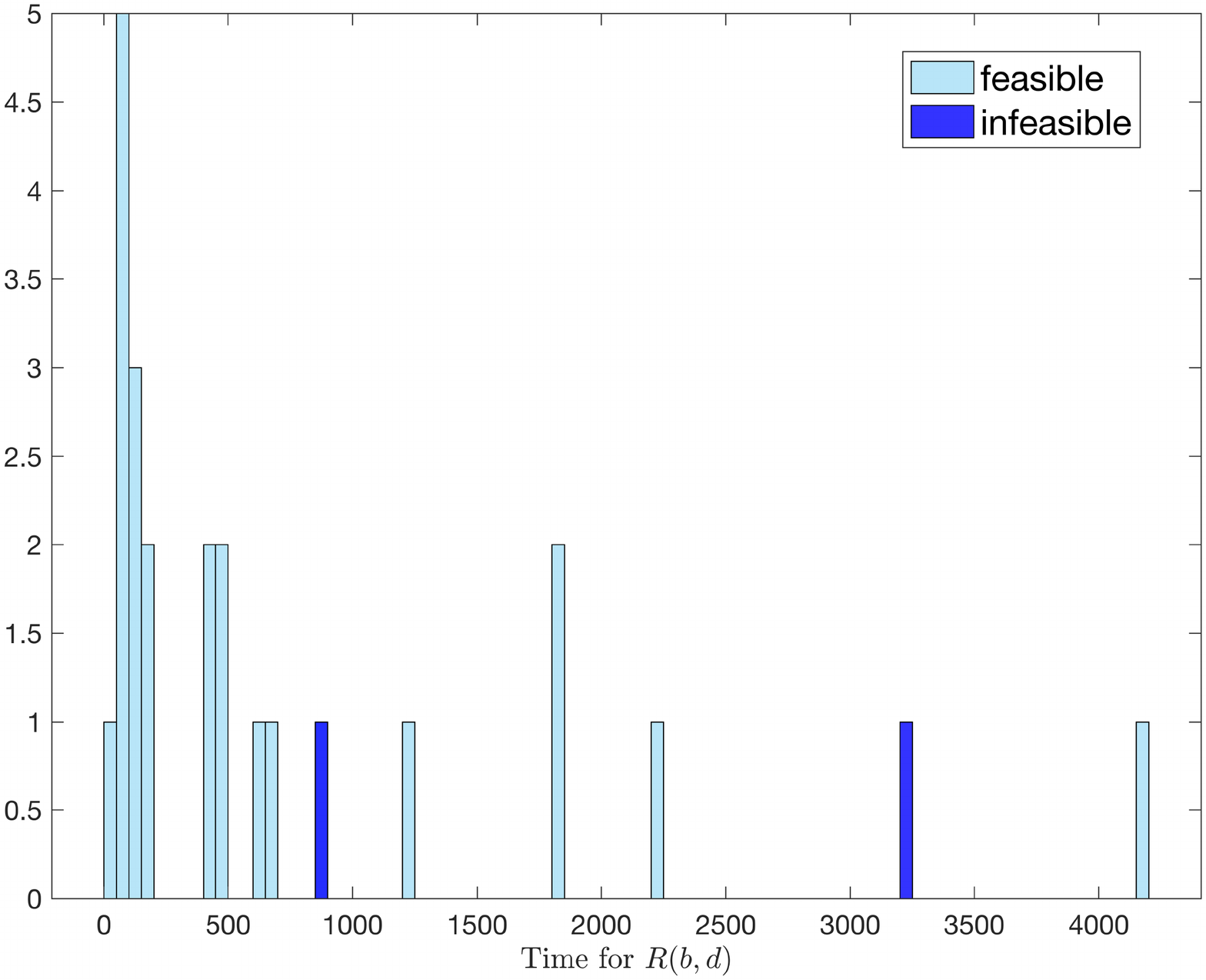}
			\caption{Occurrence frequencies of solution times that are over 30 seconds}
			\label{fig:His_out}
		\end{minipage} 		
	\end{figure} 
	
	To overcome this issue of extreme solution times, we modify CN and solve each \eqref{eq:Rbd} model under a time limit. If the model is aborted due to the time limit, we slightly modify the direction and solve the model with the new direction parameter. That is, we change line 4 of Algorithm 1 as follows\footnote{Note that the direction change is done only once.}: 
	
	\LinesNumberedHidden
	\begin{procedure}[H] \label{Proc}
		\caption{Replacement of line 4 of Algorithm 1()}
		\DontPrintSemicolon
		Let  $b(s^{b}, p^{b}, t^{b})\in B$ and $d = (p^b_1-s^b_1, t^b_2-s^b_2)^T$, attempt to solve \eqref{eq:Rbd} \;
		\If{\eqref{eq:Rbd} could not be solved within the time limit}{
			$d^{b}_2=d^{b}_2-1$\;
			Solve \eqref{eq:Rbd}\;}
	\end{procedure}
	
	We refer to this extension of CN with time limited \eqref{eq:Rbd} models as TL-CN. We compare the performance of TL-CN (where a time limit of 50 seconds is used for each \eqref{eq:Rbd} model) with those of FN and CN in class C of AP. The results are presented in Table \ref{tab:timelimitCDN}, where we report the average values of each indicator over five instances. When we compare the number of integer programming problems solved by the algorithms, we observe that CN is the best algorithm and it is closely followed by TL-CN, as expected. When we analyse the run times and average \eqref{eq:Rbd} solution times of TL-CN and CN, we observe that there is a significant improvement when TL-CN is used, indicating that the extension is successful.
	
	\begin{table}[htbp]
		\centering
		\caption{Comparison of the FN, CN and TL-CN for class C of the set AP}
		\resizebox{\textwidth}{!}{	\begin{tabular}{ccccccccc}
				
				\multirow{2}[0]{*}{Algorithm} & \multirow{2}[0]{*}{\# IP} & \multirow{2}[0]{*}{Run Time} & \multirow{2}[0]{*}{\# \eqref{eq:Rbd}} & \multirow{2}[0]{*}{Time per \eqref{eq:Rbd} } & \multirow{2}[0]{*}{Time per ($P_i(x^b)$)} & \multirow{2}[0]{*}{$ C $} & \multirow{2}[0]{*}{$ C_2 $} & \multirow{2}[0]{*}{$ E $} \\
				&       &       &       &       &       &       &       &  \\		\hline
				FN    & 1895.00 & 5644.20 & 803.60 & 5.22  & 1.32  & 288.80 & 23.00 & 794.60 \\
				TL-CN & 1850.80 & 6978.42 & 826.00   & 6.89  & 1.25  & 211.60 & 12.40 & 803.40 \\
				CN    & 1839.40 & 11212.29 & 815.80 & 12.04 & 1.33  & 210.60 & 12.60 & 803.20 \\
		\end{tabular}}%
		\label{tab:timelimitCDN}
	\end{table}%

	We also run TL-CN with predetermined time limits for class C of AP and observe the quality of the solution set (using coverage error and hypervolume gap) by comparing it with FN and CN. The average values of the number of solutions found are: 380.2, 246.2, 367.8; the average coverage error values are: 571, 40, 33.8; the average scaled coverage error values are: 0.2254, 0.0157, 0.0134, and scaled hypervolume gap values are: $18.2876\times 10^{-3},0.8229\times 10^{-3}, 0.4315\times 10^{-3}$ for FN, CN and TL-CN, respectively.  It is seen that TL-CN outperforms CN and FN in terms of representativeness. 
	
	Overall, the results show that this modification (TL-CN) is successful in significantly reducing run time without sacrificing from performance in representativeness.

	In Algorithm 1, since solving \eqref{eq:Rbd} only guarantees finding a weakly nondominated point, we rely on two integer programming problems (\eqref{eq:P1} or \eqref{eq:P2}), and in some cases there is a possibility of solving both models, one of which may be redundant. An alternative strategy would be solving a model of the following form after  \eqref{eq:Rbd}: 
	\begin{equation}\label{eq:S}
	\tag{$S(x^b)$} \text{min}\{\, z_1(x)+z_2(x)\st x\in\mathcal{X},\ z_1(x)\leq z_1(x^{b})\, \ z_2(x)\leq z_2(x^{b})\    \} 
	\end{equation}
	In this case, a single point is returned at each iteration  and the two new boxes are formed taking this point into account. That is, we modify lines 8-17 of Algorithm 1 as follows:
	\begin{procedure}[H]
		\caption{Replacement of lines 8-17 of Algorithm 1()}
		\DontPrintSemicolon		
		Solve \eqref{eq:S}.  Let $ x^{1} $ be an optimal solution. \;
		$n^{1}=z(x^{1})$   \;
		$n^{2}=z(x^{1})$   \;	
	\end{procedure}
	
	We implement this strategy for CN and refer to the variant as CN-S. We compare this variant with CN for class A of KP and AP and report the results in Table 	\ref{tab:Extensions}. The table shows that in CN-S, the number of IPs solved is less, but this decrease is not reflected on the solution times. This is because the number of \eqref{eq:Rbd}  problems, which take more time to solve, increases in CN-S compared to CN. Note that in CN-S, each nondominated point is generated after solving a separate \eqref{eq:Rbd} problem, while in some iterations of CN ($ C_2  $ iterations) two nondominated points are found after solving only one \eqref{eq:Rbd}. The redundant models CN-S avoids are the second stage models, which are easier to solve than \eqref{eq:Rbd}, hence only account for a small fraction of the total solution time. That is why, we see that the extension CN-S may not guarantee better results in terms of solution time.
	
	One key design parameter of Algorithm 1 is the order of boxes. We use a simple strategy and add the new boxes to the end of the list $\mathcal{B}$ and hence implement a First-Generated- First-Explored rule. This way, without strictly imposing it, we encourage larger boxes to be explored first. This is the case especially in CN, as the boxes are split with respect to centrally located points. To see whether there are further benefits if the ordering would always be by box size, we implement another extension of CN, in which the boxes are explored starting from the largest. This variant is called CN-BO and the results on class A of KP and AP are given in Table \ref{tab:Extensions}. The results show that there is not a significant difference between CN and CN-BO, indicating that our box exploration strategy encourages checking larger boxes first. 
	
	\begin{table}[h]
		\centering
		\caption{Comparison of CN, CN-BO and CN-S for class A of sets KP and AP}
		\resizebox{0.9\textwidth}{!}{	\begin{tabular}{ccccccccc}
				\multirow{2}[0]{*}{Set} & \multirow{2}[0]{*}{Algorithm} & \multirow{2}[0]{*}{\# IP} & \multirow{2}[0]{*}{Run Time } & \multirow{2}[0]{*}{\# \eqref{eq:Rbd}} & \multirow{2}[0]{*}{Time \eqref{eq:Rbd}} & \multirow{2}[0]{*}{$ C $} & \multirow{2}[0]{*}{$ C_2 $} & \multirow{2}[0]{*}{$ E $} \\
				&       &       &       &       &       &       &       &  \\\hline
				\multirow{3}[0]{*}{KP} & CN    & 2398.60 & 894.72 & 1351.0 & 787.13 & 72.60 & 2.40  & 592.0 \\
				& CN-BO & 2404.00 & 889.05 & 1355.2 & 773.78 & 73.60 & 2.20  & 588.2 \\
				& CN-S  & 2352.40 & 928.20 & 1375.0 & 795.12 & -     & -     & 572.8 \\
				&       &       &       &       &       &       &       &  \\
				\multirow{3}[0]{*}{AP} & CN    & 1553.40 & 2187.07 & 712.0 & 1728.28 & 140.20 & 9.20  & 683.4 \\
				& CN-BO & 1552.60 & 2178.80 & 713.0 & 1730.36 & 137.40 & 8.20  & 684.4 \\
				& CN-S  & 1434.00 & 2191.11 & 723.6 & 1735.60 & -     & -     & 690.2 \\
		\end{tabular}}%
		\label{tab:Extensions}%
	\end{table}%

	\subsection{Comparison with existing algorithms} \label{sec:ComEx}
	We also provide comparisons with existing box-searching algorithms in the literature that are reported to perform well. We coded the box algorithm proposed by \cite{hamacher2007finding} (Algorithm BA), which searches boxes using an epsilon constraint type scalarization and the algorithm discussed e.g. in  \cite{Chalmet,Ladesma,leitner2016ilp} (Algorithm WS), which uses weighted sum scalarization with box defining constraints as follows: 
	\begin{equation} \label{eq:WS}
	\tag{$WS(b,w)$}
	\text{min}\{\, wz_1(x)+(1-w)z_2(x) \st x\in\mathcal{X}, \; z_1(x) \leq \ p_1^{b} - \epsilon, \; z_2(x) \leq \ t_2^{b}- \epsilon\,\}
	\end{equation}

	We implement the algorithms to solve the same set of AP and KP using the same computational environment. To have a fair comparison, we also used our box elimination rule in all {variants of the WS algorithm. However, in BA, since the corners of the boxes are not necessarily nondominated points, our elimination strategy can not be implemented directly. Therefore, we use the elimination rule given in the original BA algorithm from \cite{hamacher2007finding}.}
	
	It is possible to implement WS with different parameter choices, each time setting the weight parameter in a similar way to the direction parameter of the Pascoletti-Serafini scalarization. Therefore, we first investigate the performance of such WS algorithm variants, in which the weight parameter $ w $ is set as follows:
	\begin{center}
		FW: $ w $=0.5, CW: $ w=\frac{t^b_2-s^b_2}{p^b_1-s^b_1+t^b_2-s^b_2} $, NW:$ w=\frac{t^0_2-s^b_2}{p^0_1-s^b_1+t^0_2-s^b_2} $
	\end{center}
	The results reported in Table \ref{tab:WSvar} show that in KP, the performances of FW, CW and NW with respect to the number of IPs is similar, while setting the weight parameters with respect to the nadir point (NW) results in a considerable increase in solution time. In these instances, FW requires the minimum solution time, which is similar to the observations made for FN (implementation of Algorithm 1 with fixed direction). In AP, setting the weight parameters with respect to the nadir point (NW) may result in an increase in the number of IP models solved, yet this implementation is the fastest one for this set. Overall, there is no clear winner over all sets in terms of solution time, FW and NW being the fastest variants over KP and AP, respectively. CW, on the other hand, seems advantageous in terms of the number of IPs solved, observed especially in AP; and is the second and the third fastest algorithm in KP and AP, respectively. By comparing Tables~\ref{tab:MainKP},~\ref{tab:MainAP} and \ref{tab:WSvar}, we observe that for each class of instances, the fastest variant of the WS algorithm outperforms CN when the goal is to compute the entire set of nondominated points.
	
	\begin{table}[h]
		\centering
		\caption{Comparison of alternative implementations of WS}
		\resizebox{\textwidth}{!}{ 
			\begin{tabular}{clccc|ccccc}	
				\multicolumn{5}{c}{KP}                & \multicolumn{5}{c}{AP} \\
				\hline
				\multirow{2}[2]{*}{Class:N} & \multicolumn{1}{c}{\multirow{2}[2]{*}{Algorithm}} & \multirow{2}[2]{*}{\# IP} & \multirow{2}[2]{*}{Run Time} & \multirow{2}[2]{*}{$ E $} & \multirow{2}[2]{*}{Class:N} & \multirow{2}[2]{*}{Algorithm} & \multirow{2}[2]{*}{\# IP} & \multirow{2}[2]{*}{Run Time} & \multirow{2}[2]{*}{$ E $\footnote{The small difference in \textit{E} values between the variants is due to numerical rounding issues.}}\\
				&       &       &       &       &       &       &       &       &  \\
				\hline
				\multirow{3}[2]{*}{A:975.4} & CW    & 1387.00 & 604.58 & 564.80 & \multirow{3}[2]{*}{A:708.4} & CW    & 729.00 & 1575.41 & 688.80 \\
				& FW    & 1383.80 & 582.13 & 568.00 &       & FW    & 738.40 & 1441.38 & 679.40 \\
				& NW    & 1384.60 & 680.64 & 567.20 &       & NW    & 747.00 & 1384.71 & 670.80 \\
				\hline
				\multirow{3}[2]{*}{B:1539.4} & CW    & 2082.40 & 1347.38 & 997.40 & \multirow{3}[2]{*}{B:1416.2} & CW    & 1538.60 & 4335.83 & 1294.80 \\
				& FW    & 2081.40 & 1313.20 & 998.40 &       & FW    & 1555.20 & 3764.10 & 1278.20 \\
				& NW    & 2080.80 & 1593.80 & 999.00 &       & NW    & 1569.40 & 3565.28 & 1264.00 \\
				\hline
				\multirow{3}[2]{*}{C:2176.2} & CW    & 2827.80 & 2297.15 & 1525.60 & \multirow{3}[2]{*}{C:823.6} & CW    & 834.80 & 3846.29 & 813.40 \\
				& FW    & 2823.00 & 2189.65 & 1530.40 &       & FW    & 847.40 & 3621.01 & 800.80 \\
				& NW    & 2825.40 & 2955.11 & 1528.00 &       & NW    & 857.00 & 3516.93 & 791.20 \\
				\hline
				\multirow{3}[2]{*}{D:2791.8} & CW    & 3472.00 & 3603.22 & 2112.60 & \multirow{3}[2]{*}{D:1827} & CW    & 1910.00 & 12056.57 & 1745.00 \\
				& FW    & 3466.80 & 3375.21 & 2117.80 &       & FW    & 1935.00 & 10371.09 & 1720.00 \\
				& NW    & 3466.40 & 5120.93 & 2118.20 &       & NW    & 1964.80 & 9816.62 & 1690.20 \\
				\hline
		\end{tabular}}%
		\label{tab:WSvar}%
	\end{table}%
	
	\begin{table}[htbp]
		\centering
		\caption{Results of BA}
		\resizebox{0.8\textwidth}{!}{ 	
			\begin{tabular}{c|cccc|cccc}
				\hline
				& \multicolumn{4}{c|}{KP}       & \multicolumn{4}{c}{AP} \\
				\hline
				Class & $ N $     & Run Time & \# IP & $ E $     & $ N $     & Run Time & \# IP & $ E $ \\
				\hline
				A     & 975.4 & 389.44 & 2703.20 & 975.40 & 708.4 & 947.53 & 1446.80 & 708.40 \\
				B     & 1539.4 & 834.10 & 4080.40 & 1539.20 & 1416.2 & 2257.25 & 3012.00 & 1416.20 \\
				C     & 2176.2 & 1373.00 & 5420.00 & 2176.20 & 823.6 & 2523.18 & 1661.20 & 823.60 \\
				D     & 2791.8 & 2297.09 & 6575.20 & 2791.80 & 1827  & 6916.81 & 3763.20 & 1827.00 \\
				\hline	\end{tabular}}%
		\label{tab:Existingfull}%
	\end{table}%
	
	Table \ref{tab:Existingfull} shows the results of the experiments performed with BA. When the solution times required to find all the nondominated solutions are considered, it is observed that in KP, 
	BA outperforms all other algorithms. The differences are more significant in AP. Since BA solves two models for each scalarization (to guarantee obtaining a nondominated solution), the number of IPs is about twice that of WS variants and comparable to those of FN and CN.
	
	One of the main motivations to use the Pascoletti-Serafini scalarization is to obtain representative subsets in short time; hence we also compare the performance of the algorithms under time limit. We perform a new set of experiments to compare CN, the best variant in terms of representativeness, with {WS variants} and BA. 
	
	For each set of instances, we perform runs setting two time limits: T\textsubscript{1} and T\textsubscript{2}. We set the time limits T\textsubscript{1} and T\textsubscript{2} such that they are a quarter of and half of the average solution time of the fastest algorithm (BA), respectively \footnote{Specifically, T\textsubscript{1} is set as 97.5, 208.5, 343 and 574 second in classes A, B, C and D of KP, respectively.  T\textsubscript{1} is set as 237, 565, 632 and 1729 seconds in classes A, B, C and D of AP, respectively. T\textsubscript{2}=2$\times$T\textsubscript{1} for each class.}. 
	
	{First, we compare the performances of variants of WS algorithm in terms of representativeness. Table~\ref{tab:WSlim} shows the average values of the number of solutions found, coverage error, scaled coverage error and scaled hypervolume gap for CW, FW and NW. Accordingly, CW, in which the weight parameter is set using a similar strategy as in CN, outperforms the others in terms of representativeness. Hence, only CW is used for the overall comparison. Note that empty boxes (hence infeasible models) are encountered towards the later iterations of the algorithm CW, compared to the other variants. As a result, CW returns more solutions than the other variants when implemented under time limit, which leads to better performance in representativeness.  }
	
	In Table \ref{tab:Existinglim}, we report the average values of the number of solutions found, coverage error, scaled coverage error and scaled hypervolume gap for CN, CW and BA. We observe that the performances of CN and CW are comparable, while BA performs significantly worse than the other two in terms of coverage. Moreover, this good performance of CN is obtained with a significantly less number of solutions compared to CW and BA (nearly half in most cases). In that sense, one can see that our algorithm is competitive as it ensures comparable level of coverage as CW with a subset of much smaller size and better coverage than BA, again with a much smaller set of solutions, which is desirable (\cite{Sayin2000}). We also observe that scaled coverage error and the scaled hypervolume gap do not provide the same order, see e.g. AP(T\textsubscript{1}) class C results, where SCE values for CN and BA are 0.0253 and 0.0303, indicating better performance for CN, while the SHG values are 2.4494 ($ \times 10^{-3} $) and 1.4997 ($ \times 10^{-3} $), indicating the opposite order. When representativess is measured in terms of SHG, we observe that CN performs worst in some sets, which is mainly due to it returning a small number of solutions. 
	
	The results indicate that when the algorithms are implemented so as to find a subset of fixed cardinality, CN may present a more representative set. This is because, unlike Algorithm 1, the WS algorithm (and hence CW) without the box constraints fails to generate unsupported points. It can only reach  unsupported points by making them locally supported, iteratively. Hence when the centrally located point(s) of a box are unsupported, they may not be generated in the early iterations. This can be illustrated in Figure \ref{fig:dispro}, where implementing CW  would return point 1, which is clearly less representative of the set of solutions in the box compared to point 2, the point returned by CN. Algorithm 1 is able to return centrally located points whether they are supported or unsupported, which gives CN a clear advantage in terms of representativeness under fixed cardinality. This can also be seen in Table \ref{tab:ExistinglimFixedN}, in which we show the representativeness measures by fixing the cardinality of the solution sets of CW and BA to the one returned by CN.


	\begin{table}[htbp]
		\centering
		\caption{{Representativeness results of the alternative implementations of WS under time limit}}
		\resizebox{\textwidth}{!}{ 
			\begin{tabular}{cc|cccc|cccc|cccc}
				\hline
				&       & \multicolumn{4}{c|}{CW}        & \multicolumn{4}{c|}{FW}        & \multicolumn{4}{c}{NW} \\
				Problem Type & Class & $ \bar{{N}} $ & CE & SCE & SHG$\times 10^3$ & $ \bar{{N}} $ & CE & SCE & SHG$\times 10^3$ & $ \bar{{N}} $ & CE & SCE & SHG$\times 10^3$ \\ \hline
				\multicolumn{1}{c}{\multirow{4}[0]{*}{AP (T\textsubscript{1})}} & A     & 171.4 & 56.6  & 0.0214 & 1.3010 & 94.6  & 1098.6 & 0.4165 & 165.4009 & 110.2 & 1081.8 & 0.4102 & 163.2202 \\
				& B     & 328   & 129   & 0.0163 & 0.4655 & 158   & 3486.4 & 0.4403 & 131.0424 & 167   & 3470.8 & 0.4383 & 130.2929 \\
				& C     & 203.4 & 51.4  & 0.0204 & 1.4052 & 116.8 & 1031.8 & 0.4074 & 186.0266 & 136.8 & 1011.6 & 0.3993 & 181.2243 \\
				& D     & 465.2 & 123.6 & 0.0127 & 0.3070 & 206.2 & 4299.6 & 0.4417 & 129.8451 & 243.4 & 4290.6 & 0.4408 & 126.3918 \\ \hline
				\multicolumn{1}{c}{\multirow{4}[0]{*}{AP (T\textsubscript{2})}} & A     & 287.8 & 45.6  & 0.0172 & 0.6642 & 192   & 1065.8 & 0.4040 & 133.0541 & 218.6 & 1032.2 & 0.3914 & 123.9464 \\
				& B     & 540   & 96.8  & 0.0122 & 0.2496 & 326   & 3402.4 & 0.4297 & 107.5144 & 356.2 & 3368.8 & 0.4255 & 103.1515 \\
				& C     & 348.4 & 34.6  & 0.0137 & 0.7168 & 240.2 & 984.6 & 0.3887 & 144.3268 & 275.6 & 955.2 & 0.3770 & 132.3241 \\
				& D     & 766.4 & 84    & 0.0087 & 0.1590 & 446.4 & 4203.6 & 0.4318 & 105.2014 & 529.8 & 4163.0 & 0.4277 & 97.0106 \\ \hline
				\multicolumn{1}{c}{\multirow{4}[0]{*}{KP (T\textsubscript{1})}} & A     & 358.2 & 44.2  & 0.0128 & 1.6565 & 159.2 & 1135.6 & 0.3273 & 249.0571 & 128.8 & 1145.4 & 0.3302 & 279.1742 \\
				& B     & 547.8 & 46.4  & 0.0098 & 1.2366 & 240.8 & 1534.8 & 0.3257 & 257.4887 & 182.8 & 1538.2 & 0.3266 & 287.8895 \\
				& C     & 760.8 & 49.2  & 0.0082 & 0.9845 & 395.6 & 1977  & 0.3279 & 247.1719 & 284.4 & 1994.6 & 0.3308 & 278.7227 \\
				& D     & 983.2 & 51.8  & 0.0073 & 0.7693 & 567.4 & 2334.8 & 0.3264 & 240.3115 & 342.8 & 2413.2 & 0.3375 & 281.7669 \\ \hline
				\multicolumn{1}{c}{\multirow{4}[0]{*}{KP (T\textsubscript{2})}} & A     & 531   & 37.2  & 0.0108 & 0.8544 & 293   & 1040.2 & 0.2993 & 188.8417 & 246.8 & 1055.2 & 0.3037 & 215.4790 \\
				& B     & 838.4 & 39.6  & 0.0084 & 0.5830 & 460.8 & 1396.2 & 0.2960 & 187.5788 & 356.2 & 1427.0 & 0.3025 & 223.4901 \\
				& C     & 1154.2 & 36.4  & 0.0061 & 0.5068 & 733.8 & 1792  & 0.2969 & 173.8502 & 514.2 & 1880.4 & 0.3118 & 222.0448 \\
				& D     & 1514.4 & 40.6  & 0.0057 & 0.4011 & 1093.2 & 2060  & 0.2879 & 155.1007 & 640.8 & 2264.8 & 0.3168 & 227.6428 \\ \hline
		\end{tabular}}%
		\label{tab:WSlim}%
	\end{table}%

	\begin{table}[!h]
		\centering
		\caption{Representativeness results of the existing algorithms under time limit}
		\resizebox{\textwidth}{!}{ 
			\begin{tabular}{cc|cccc|cccc|cccc} \hline
				&       & \multicolumn{4}{c|}{CN}        & \multicolumn{4}{c|}{CW}        & \multicolumn{4}{c}{BA} \\
				Problem Type & Class & $ \bar{{N}} $ & CE & SCE & SHG$\times 10^3$ & $ \bar{{N}} $ & CE & SCE & SHG$\times 10^3$ & $ \bar{{N}} $ & CE & SCE & SHG$\times 10^3$ \\ \hline
				\multicolumn{1}{c}{\multirow{4}[0]{*}{AP (T\textsubscript{1})}} & A     & 95.8  & 80.8  & 0.0306 & 2.0348 & 171.4 & 56.6  & 0.0214 & 1.3010 & 186.2 & 189.2 & 0.0718 & 1.8700 \\
				& B     & 209.2 & 156   & 0.0197 & 0.5574 & 328   & 129   & 0.0163 & 0.4655 & 388   & 473.2 & 0.0596 & 0.4790 \\
				& C     & 100.8 & 64    & 0.0253 & 2.4494 & 203.4 & 51.4  & 0.0204 & 1.4052 & 217.8 & 76.8  & 0.0303 & 1.4997 \\
				& D     & 283.4 & 135.2 & 0.0139 & 0.3622 & 465.2 & 123.6 & 0.0127 & 0.3070 & 493.2 & 505.4 & 0.0519 & 0.4531 \\ \hline
				\multicolumn{1}{c}{\multirow{4}[0]{*}{AP (T\textsubscript{2})}} & A     & 190.6 & 46    & 0.0175 & 0.8295 & 287.8 & 45.6  & 0.0172 & 0.6642 & 363.8 & 48.2  & 0.0183 & 0.3556 \\
				& B     & 405.2 & 104   & 0.0132 & 0.2242 & 540   & 96.8  & 0.0122 & 0.2496 & 737   & 226.6 & 0.0285 & 0.1606 \\
				& C     & 166.6 & 45.6  & 0.0180 & 1.3423 & 348.4 & 34.6  & 0.0137 & 0.7168 & 417.2 & 33    & 0.0130 & 0.3609 \\
				& D     & 513.8 & 69.8  & 0.0072 & 0.1466 & 766.4 & 84    & 0.0087 & 0.1590 & 932.8 & 146.6 & 0.0150 & 0.0938 \\ \hline
				\multicolumn{1}{c}{\multirow{4}[0]{*}{KP (T\textsubscript{1})}} & A     & 225.6 & 45.8  & 0.0135 & 1.5178 & 358.2 & 44.2  & 0.0128 & 1.6565 & 353.2 & 72.4  & 0.0210 & 0.8235 \\
				& B     & 309.6 & 39.6  & 0.0082 & 1.2757 & 547.8 & 46.4  & 0.0098 & 1.2366 & 564.2 & 57.2  & 0.0123 & 0.5105 \\
				& C     & 325.8 & 51.4  & 0.0085 & 1.2748 & 760.8 & 49.2  & 0.0082 & 0.9845 & 739.4 & 64.6  & 0.0107 & 0.4238 \\
				& D     & 376.6 & 58    & 0.0081 & 1.1317 & 983.2 & 51.8  & 0.0073 & 0.7693 & 1007.6 & 72.8  & 0.0102 & 0.2993 \\ \hline
				\multicolumn{1}{c}{\multirow{4}[0]{*}{KP (T\textsubscript{2})}} & A     & 423.2 & 22.2  & 0.0064 & 0.5189 & 531   & 37.2  & 0.0108 & 0.8544 & 637.2 & 39.4  & 0.0116 & 0.1881 \\
				& B     & 643.8 & 23.6  & 0.0048 & 0.4171 & 838.4 & 39.6  & 0.0084 & 0.5830 & 1019.4 & 27.2  & 0.0059 & 0.1182 \\
				& C     & 650.6 & 21.4  & 0.0036 & 0.4839 & 1154.2 & 36.4  & 0.0061 & 0.5068 & 1412.6 & 30    & 0.0050 & 0.1034 \\
				& D     & 758.2 & 25.6  & 0.0036 & 0.4485 & 1514.4 & 40.6  & 0.0057 & 0.4011 & 1804.8 & 25.2  & 0.0035 & 0.0765 \\ \hline
		\end{tabular}}%
		\label{tab:Existinglim}%
	\end{table}%

	\begin{table}[!h]
		\centering
		\caption{Representativeness results of the existing algorithms for fixed $\bar{N}$}
		\resizebox{0.8\textwidth}{!}{ 
			\begin{tabular}{cc|c|ccc|ccc} \hline
				&       &       & \multicolumn{3}{c|}{CW} & \multicolumn{3}{c}{BA} \\
				Problem Type & Class & $\bar{N}$ & CE & SCE & SHG$\times 10^3$ & CE & SCE & SHG$\times 10^3$ \\\hline
				\multicolumn{1}{c}{\multirow{4}[0]{*}{AP (T\textsubscript{1})}} & A     & 95.8  & 88    & 0.0335 & 2.4701 & 320.8 & 0.1219 & 4.2239 \\
				& B     & 209.2 & 163.8 & 0.0207 & 0.7288 & 841   & 0.1058 & 1.4844 \\
				& C     & 100.8 & 79.6  & 0.0315 & 2.8475 & 223.4 & 0.0883 & 4.2464 \\
				& D     & 283.4 & 153.2 & 0.0157 & 0.5073 & 878.4 & 0.0902 & 1.0034 \\\hline
				\multicolumn{1}{c}{\multirow{4}[0]{*}{AP (T\textsubscript{2})}} & A     & 190.6 & 56.6  & 0.0214 & 1.1732 & 154.6 & 0.0591 & 1.7069 \\
				& B     & 405.2 & 129   & 0.0163 & 0.3622 & 473.2 & 0.0596 & 0.4609 \\
				& C     & 166.6 & 66    & 0.0262 & 1.7169 & 131   & 0.0518 & 2.5594 \\
				& D     & 513.8 & 91.4  & 0.0094 & 0.2694 & 505.4 & 0.0519 & 0.4503 \\\hline
				\multicolumn{1}{c}{\multirow{4}[0]{*}{KP (T\textsubscript{1})}} & A     & 225.6 & 51.2  & 0.0147 & 2.9080 & 82.4  & 0.0241 & 1.7554 \\
				& B     & 309.6 & 67.8  & 0.0144 & 2.4493 & 103.2 & 0.0218 & 1.3959 \\
				& C     & 325.8 & 74.6  & 0.0124 & 2.2843 & 133.8 & 0.0222 & 1.3735 \\
				& D     & 376.6 & 70.2  & 0.0099 & 1.9943 & 174.2 & 0.0244 & 1.2145 \\\hline
				\multicolumn{1}{c}{\multirow{4}[0]{*}{KP (T\textsubscript{2})}} & A     & 423.2 & 39.4  & 0.0114 & 1.2623 & 54.6  & 0.0158 & 0.5883 \\
				& B     & 643.8 & 46.4  & 0.0098 & 1.0227 & 66.6  & 0.0141 & 0.4626 \\
				& C     & 650.6 & 55.8  & 0.0092 & 1.2001 & 75.4  & 0.0125 & 0.5125 \\
				& D     & 758.2 & 55    & 0.0077 & 1.0475 & 110.8 & 0.0155 & 0.4763 \\ \hline
		\end{tabular}}%
		\label{tab:ExistinglimFixedN}%
	\end{table}%

	There is also a recent algorithm, the balanced box (BB) algorithm (\cite{Boland}), which can be seen as an extension of \cite{hamacher2007finding}.  Since this algorithm is originally coded using a different language and involves various enhancements, it is difficult to replicate the results in our computational environment. Therefore, we do not fully compare the solution times with those of the balanced box algorithm. We, however, can comment on the number of (mixed) integer programming problems solved. The balanced box algorithm solves exactly $3N$ problems. Therefore, it will solve more models for all of the problem instances considered; indeed it solves 25.5\%, 36.5\% more problems than our best algorithm variant on average for KP and AP, respectively. Moreover, on a different computer on which we could run the code of BB algorithm provided by one of the authors of \cite{Boland} and our algorithms, we 
	made comparisons between BB and CN when both are implemented under time limit. Specifically, for each instance in class B and D of sets KP and AP, we first run BB algorithm  until a fixed number of nondominated points are found (100, 200 and 300 points, respectively) and record the solution time. We then run CN under the same time limit and report detailed results on the number of nondominated points found, the coverage error values (CE and SCE) and the scaled hypervolume gap (SHG) for both algorithms in Tables \ref{tab:BBTypeBcov} and \ref{tab:BBTypeDcov}, for class B and D instances, respectively (See Appendix). The results indicate better performance of the CN algorithm with respect to representativeness. In almost all instances, CN provides a solution set with better coverage and hypervolume gap, while in a few instances the coverage of the solution set of BB is better. We use Figure \ref{fig:TimeEffect} to summarize these results and demonstrate how the quality of approximation evolves over time. The figure shows the change in SCE and SHG values through time for BB and CN. We see that CN provides better representativeness under all three time limits and that the difference reduces as time limit increases.
	
	
		\begin{figure}[h]						
				\centering  
	\includegraphics[scale=0.7]{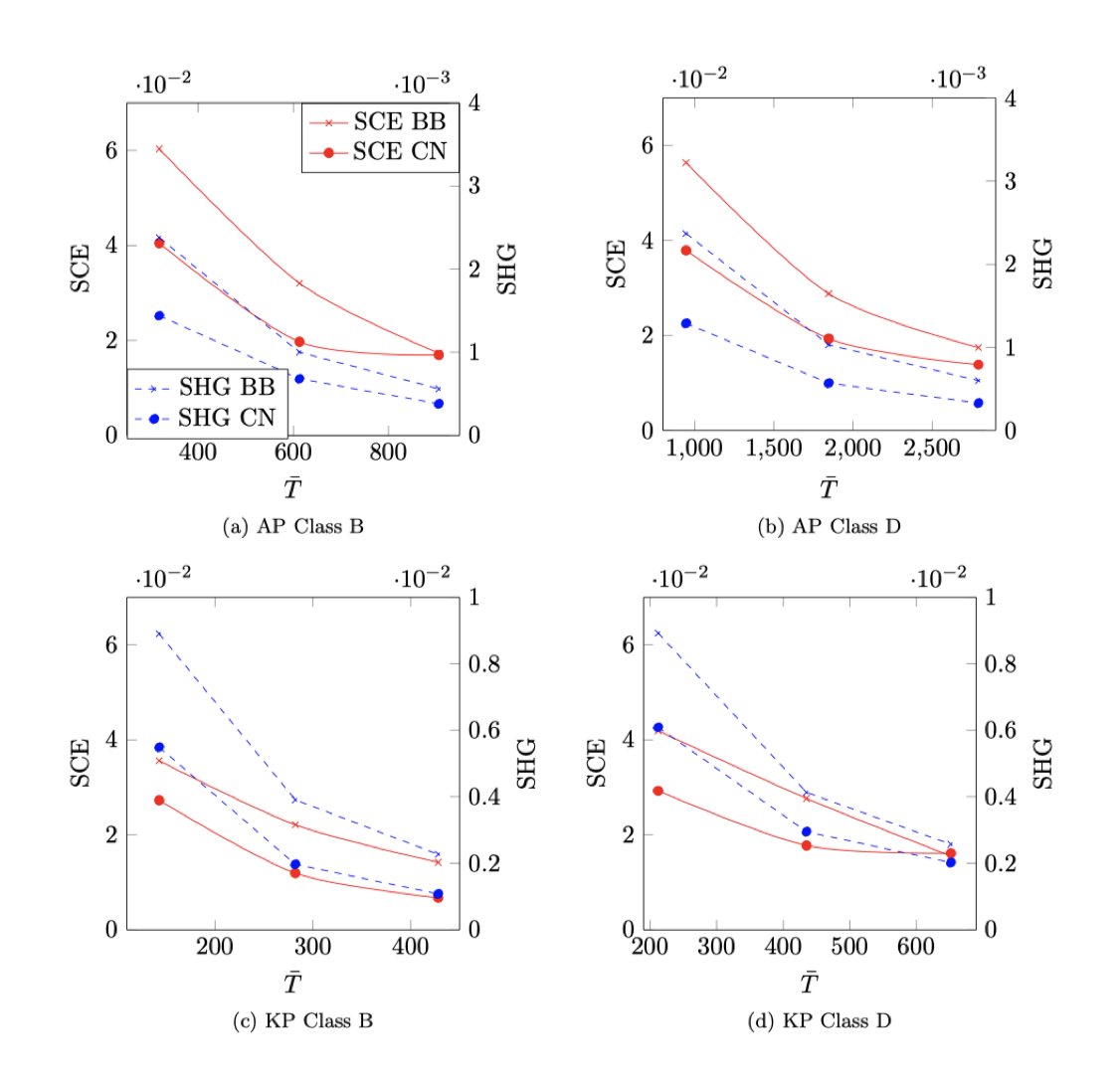}
		\caption{Quality of the approximation over time. ($\bar{T}$: Time limit)}
		\label{fig:TimeEffect}
		\end{figure}

	\section{Conclusion} \label{sec:Conc}
	
	We propose an exact solution approach for biobjective integer programming problems based on  solving Pascoletti-Serafini scalarizations to search for nondominated points within boxes in the objective space. We implement different variations of the algorithm based on how the boxes are defined and how the direction vector in the scalarization problem is set.  We prove that the algorithm terminates and provide lower and upper bounds on the number of scalarization models solved.
	
	We compare the performances of alternative implementations of the algorithm under different parameter choices and box splitting strategies, both with and without time limits. Our results indicate that using nondominated points to define the boxes is a better strategy. Moreover, although using a fixed direction vector of  $(1,1)^T$ leads to more (mixed) integer programming problems solved, it requires less computational time since less of the more difficult scalarization models are solved. We, however, observe that setting direction with respect to the diagonal of the box to be searched is still promising since it returns a highly representative subset (measured using coverage error and hypervolume gap) of the set of nondominated points when it is run with a time limit. This good performance in representativeness is also verified by comparisons with existing algorithms.
	
	Future research can focus on further comparisons of the discussed approaches on different sets of problems as well as exploration of different versions of Algorithm 1, one example being the implementation of the augmented formulation of the scalarization problem instead of solving second stage models. 
	
	\bibliographystyle{plain} 
	\bibliography{whole} 

	\newpage
	\section*{Appendix A: Detailed results on quality of approximation}

	\begin{table}[h!]
		\centering
		\caption{Representativeness results of BB and CN under time limit for class B instances}
		\resizebox{\textwidth}{!}{
			\begin{tabular}{c|cc|cccc|cccc|cc|cccc|cccc}
				\hline
				& \multicolumn{10}{c|}{KP}                                                      & \multicolumn{10}{c}{AP} \\
				&       & \multicolumn{1}{c}{} & \multicolumn{4}{c}{BB}        & \multicolumn{4}{c|}{CN}       &       & \multicolumn{1}{c}{} & \multicolumn{4}{c}{BB}        & \multicolumn{4}{c}{CN} \\\hline
				&  &    &   &   &  &  &     &    &  &   &   &   &    &   &  & 
				&   &    & & 
				\\
				Ins   & $ N $     & Time  & $  \bar{N} $     & CE    & SCE   & SHG$\times 10^3$  & $  \bar{N} $     & CE   & SCE   & SHG$\times 10^3$   & $ N $     & Time  & $  \bar{N} $     & CE    & SCE   & SHG$\times 10^3$  & $  \bar{N} $     & CE    & SCE   & SHG$\times 10^3$ \\\hline	
				&  &    &   &   &  &  &     &    &  &   &   &   &    &   &  & 
				&   &    & & 
				\\
				1     & 1416  & 135   & 100   & 186   & 0.0433 & 8.6553 & 97    & 124   & 0.0289 & 4.8684 & 1508  & 321   & 100   & 461   & 0.0567 & 2.4728 & 94    & 368   & 0.0452 & 1.4702 \\
				2     & 1793  & 167   & 100   & 217   & 0.0405 & 9.0773  & 78    & 149   & 0.0278 & 6.4003 & 1362  & 317   & 100   & 485   & 0.0631 & 2.3072 & 96    & 275   & 0.0358 & 1.3666 \\
				3     & 1539  & 164   & 100   & 144   & 0.0300 & 9.5240 & 96    & 130   & 0.0271 & 5.3799 & 1383  & 315   & 100   & 576   & 0.0724 & 2.2801 & 94    & 391   & 0.0492 & 1.3933 \\
				4     & 1338  & 127   & 100   & 127   & 0.0282 & 8.5474 & 96    & 144   & 0.0320 & 5.0485 & 1456  & 325   & 100   & 462   & 0.0596 & 2.5565 & 93    & 288   & 0.0371 & 1.5630 \\
				5     & 1611  & 123   & 100   & 165   & 0.0359 & 8.9997 & 87    & 94    & 0.0205 & 5.7980 & 1372  & 313   & 100   & 401   & 0.0497 & 2.3082 & 93    & 283   & 0.0351 & 1.4088 \\
				&       &       &       &       &       &       &       &       &       &       &       &       &       &       &       &       &       &       &       &  \\				
				1     & 1416  & 261   & 200   & 134   & 0.0312 & 3.7353 & 258   & 42    & 0.0098 & 1.2703 & 1508  & 614   & 200   & 199   & 0.0245 & 1.0482 & 183   & 183   & 0.0225 & 0.6781 \\
				2     & 1793  & 353   & 200   & 150   & 0.028 & 4.0634 & 197   & 92    & 0.0172 & 2.3858 & 1362  & 615   & 200   & 296   & 0.0385 & 0.9560 & 182   & 141   & 0.0184 & 0.6460 \\
				3     & 1539  & 333   & 200   & 86    & 0.0179 & 4.2037 & 235   & 45    & 0.0094 & 1.7662 & 1383  & 612   & 200   & 313   & 0.0394 & 0.9515 & 177   & 131   & 0.0165 & 0.6591 \\
				4     & 1338  & 243   & 200   & 68    & 0.0151 & 3.6248 & 215   & 52    & 0.0115 & 1.8965 & 1456  & 624   & 200   & 213   & 0.0275 & 1.0827 & 177   & 164   & 0.0211 & 0.7440 \\
				5     & 1611  & 222   & 200   & 84    & 0.0183 & 3.9165 & 179   & 54    & 0.0118 & 2.5501 & 1372  & 601   & 200   & 244   & 0.0303 & 0.9612 & 179   & 161   & 0.02  & 0.6586 \\
				&       &       &       &       &       &       &       &       &       &       &       &       &       &       &       &       &       &       &       &  \\
				1     & 1416  & 377   & 300   & 60    & 0.0140 & 2.1089 & 370   & 42    & 0.0098 & 0.8973 & 1508  & 898   & 300   & 161   & 0.0198 & 0.5926 & 261   & 106   & 0.0130 & 0.3810 \\
				2     & 1793  & 537   & 300   & 114   & 0.0213 & 2.4171 & 320   & 54    & 0.0101 & 1.2813 & 1362  & 916   & 300   & 141   & 0.0184 & 0.5302 & 273   & 108   & 0.0141 & 0.3395 \\
				3     & 1539  & 493   & 300   & 49    & 0.0102 & 2.4582 & 366   & 21    & 0.0044 & 1.0527 & 1383  & 909   & 300   & 126   & 0.0158 & 0.5318 & 251   & 131   & 0.0165 & 0.3707 \\
				4     & 1338  & 373   & 300   & 57    & 0.0127 & 2.1271 & 359   & 21    & 0.0047 & 0.9828 & 1456  & 919   & 300   & 98    & 0.0126 & 0.6115 & 252   & 164   & 0.0211 & 0.4214 \\
				5     & 1611  & 358   & 300   & 59    & 0.0128 & 2.2552 & 318   & 22    & 0.0048 & 1.2060 & 1372  & 890   & 300   & 161   & 0.0200 & 0.5405 & 252   & 161   & 0.0200 & 0.3803 \\
				\hline
		\end{tabular}}%
		\label{tab:BBTypeBcov}%
	\end{table}%
	
	\begin{table}[h!]
		\centering
		\caption{Representativeness results of BB and CN under time limit for class D instances}
		\resizebox{\textwidth}{!}{
			\begin{tabular}{c|cc|cccc|cccc|cc|cccc|cccc}
				\hline
				& \multicolumn{10}{c|}{KP}                                                      & \multicolumn{10}{c}{AP} \\
				&       & \multicolumn{1}{c}{} & \multicolumn{4}{c}{BB}        & \multicolumn{4}{c|}{CN}       &       & \multicolumn{1}{c}{} & \multicolumn{4}{c}{BB}        & \multicolumn{4}{c}{CN} \\
				\hline
				&  &    &   &   &  &  &     &    &  &   &   &   &    &   &  & 
				&   &    & & 
				\\
				Ins   & $ N $     & Time  & $  \bar{N} $     & CE    & SCE   & SHG$\times 10^3$   & $  \bar{N} $    & CE    & SCE   & SHG $\times 10^3$  & $ N $     & Time  &$  \bar{N} $     & CE    & SCE   & SHG  $\times 10^3$ &$  \bar{N} $     & CE    & SCE  & SHG $\times 10^3$ \\\hline
				&  &    &   &   &  &  &     &    &  &   &   &   &    &   &  & 
				&   &    & & 
				\\
				1     & 3030  & 209   & 100   & 286   & 0.0396 & 9.3300 & 74    & 218   & 0.0301 & 6.6541  & 1746  & 904   & 100   & 528   & 0.0532 & 2.2107
				& 103   & 362   & 0.0365 & 1.2164
				\\
				2     & 2836  & 230   & 100   & 274   & 0.0397 & 9.2092 & 74    & 218   & 0.0316 & 6.6024 & 1889  & 939   & 100   & 550   & 0.0564 & 2.5445		 & 104   & 412   & 0.0422 & 1.3476
				
				\\
				3     & 2920  & 216   & 100   & 339   & 0.0442 & 8.9635 & 87    & 206   & 0.0269 & 5.7109 & 1834  & 938   & 100   & 556   & 0.0576 & 2.3210		 & 103   & 352   & 0.0365 & 1.2662
				
				\\
				4     & 2686  & 205   & 100   & 271   & 0.0385 & 8.6333 & 67    & 214   & 0.0304 & 6.5728 & 1839  & 952   & 100   & 541   & 0.0558 & 2.3854		 & 96    & 383   & 0.0395 & 1.4014
				
				\\
				5     & 2487  & 199   & 100   & 329   & 0.0475 & 8.4906 & 94    & 188   & 0.0272 & 4.9580 & 1827  & 997   & 100   & 569   & 0.0589 & 2.4342 & 106   & 337   & 0.0349 & 1.2618  \\
				&       &       &       &       &       &       &       &       &       &       &       &       &       &       &       &       &       &       &       &  \\
				1     & 3030  & 428   & 200   & 189   & 0.0261 & 4.3571 & 139   & 133   & 0.0184 & 3.2987 & 1746  & 1783  & 200   & 299   & 0.0301 & 0.9477
				& 203   & 185   & 0.0186 & 0.5560
				
				\\
				2     & 2836  & 445   & 200   & 196   & 0.0284 & 4.2477 & 132   & 123   & 0.0178 & 3.3739 & 1889  & 1836  & 200   & 273   & 0.028 & 1.1116		 & 204   & 207   & 0.0212 & 0.6304
				
				\\
				3     & 2920  & 467   & 200   & 211   & 0.0275 & 4.1654 & 172   & 118   & 0.0154 & 2.7265 & 1834  & 1827  & 200   & 280   & 0.029 & 1.0027		 & 213   & 181   & 0.0188 & 0.5409
				
				\\
				4     & 2686  & 424   & 200   & 190   & 0.027 & 4.0136 & 145   & 148   & 0.0211 & 3.0111 & 1839  & 1858  & 200   & 269   & 0.0278 & 1.0353		 & 204   & 178   & 0.0184 & 0.5920
				
				\\
				5     & 2487  & 411   & 200   & 202   & 0.0292 & 3.8847 & 189   & 111   & 0.016 & 2.3157 & 1827  & 1924  & 200   & 282   & 0.0292 & 1.0571 & 212   & 190   & 0.0197 & 0.5697 \\
				&       &       &       &       &       &       &       &       &       &       &       &       &       &       &       &       &       &       &       &  \\
				
				1     & 3030  & 664   & 300   & 114   & 0.0158 & 2.7535 & 230   & 133   & 0.0184 & 1.9805 & 1746  & 2905  & 300   & 198   & 0.02  & 0.5555
				& 315   & 166   & 0.0167 & 0.3083
				\\
				2     & 2836  & 689   & 300   & 91    & 0.0132 & 2.6506 & 205   & 123   & 0.0178 & 2.3463
				& 1889  & 2725  & 300   & 160   & 0.0164 & 0.6500		 & 309   & 122   & 0.0125 & 0.3653
				
				\\
				3     & 2920  & 643   & 300   & 115   & 0.015 & 2.5678 & 221   & 118   & 0.0154 & 2.0427 & 1834  & 2695  & 300   & 167   & 0.0173 & 0.5804		 & 323   & 136   & 0.0141 & 0.3155
				\\
				4     & 2686  & 647   & 300   & 93    & 0.0132 & 2.4708 & 191   & 148   & 0.0211 & 2.3945 & 1839  & 2785  & 300   & 168   & 0.0173 & 0.6122
				& 307   & 130   & 0.0134 & 0.3425
				\\		
				5     & 2487  & 616   & 300   & 142   & 0.0205 & 2.3954 & 282   & 55    & 0.0079 & 1.3363 & 1827  & 2847  & 300   & 158   & 0.0164 & 0.6166 & 322   & 122   & 0.0126 &0.3335 \\
				\hline
		\end{tabular}}%
		\label{tab:BBTypeDcov}%
	\end{table}%

\end{document}